\documentclass[journal]{IEEEtran}

\ifCLASSINFOpdf
   \usepackage[pdftex]{graphicx}
\else
\fi
\usepackage{amsmath,amsfonts,amsthm}

\newtheorem{theorem}{Theorem}
\newtheorem{lemma}{Lemma}
\theoremstyle{remark}
\newtheorem{rem}{Remark}

\usepackage[normalem]{ulem}
\newcommand\redout{\bgroup\markoverwith
{\textcolor{red}{\rule[.5ex]{2pt}{0.4pt}}}\ULon}
\usepackage{xcolor}

\usepackage{algorithm,algorithmic}
\begin{document}
%
\title{Distributed Banach-Picard Iteration: Application to Distributed EM and Distributed PCA}
\title{Distributed Banach-Picard Iteration: Application to Distributed Parameter Estimation and PCA}
%
%
%

\author{Francisco~Andrade,
        M\'{a}rio~A.~T.~Figueiredo,~\IEEEmembership{Fellow,~IEEE},
        and Jo\~{a}o~Xavier
\thanks{The authors are with the Instituto Superior T\'{e}cnico, Universidade de Lisboa, Lisbon, Portugal.} %
\thanks{F.~Andrade (francisco.andrade@tecnico.ulisboa.pt) and M.~Figueiredo (mario.figueiredo@tecnico.ulisboa.pt) are also with the Instituto de Telecomunicaç\~{o}es, Lisbon, Portugal.}
\thanks{J.~Xavier (jxavier@isr.tecnico.ulisboa.pt) is  also with the Laboratory for Robotics and Engineering Systems, Institute for Systems and Robotics, Lisbon, Portugal.}}

%
%

\markboth{SUBMITTED}%
{Shell \MakeLowercase{\textit{et al.}}: Bare Demo of IEEEtran.cls for IEEE Journals}
%



\maketitle

\begin{abstract}
In recent work, we proposed a \emph{distributed Banach-Picard iteration} (DBPI) that allows a set of agents, linked by a communication network, to find a fixed point of a \emph{locally contractive} (LC) map that is the average of individual maps held by said agents. In this work, we build upon the DBPI and its \emph{local linear convergence} (LLC) guarantees to make several contributions. We show that Sanger’s algorithm for \emph{principal component analysis} (PCA) corresponds to the iteration of an LC map that can be written as the average of local maps, each map known to each agent holding a subset of the data. Similarly, we show that a variant of the \emph{expectation-maximization} (EM) algorithm for parameter estimation from noisy and faulty measurements in a sensor network can be written as the iteration of an LC map that is the average of local maps, each available at just one node. Consequently, via the DBPI, we derive two distributed algorithms – distributed EM and distributed PCA – whose LLC guarantees follow from those that we proved for the DBPI. The verification of the LC condition for EM is challenging, as the underlying operator depends on random samples, thus the LC condition is of probabilistic nature.
\end{abstract}

\begin{IEEEkeywords}
Distributed Computation, Banach-Picard Iteration, Fixed Points, Distributed EM, Distributed PCA, Consensus.
\end{IEEEkeywords}

%
\IEEEpeerreviewmaketitle

\section{Introduction}

\IEEEPARstart{P}{arameter} estimation from noisy data and dimensionality reduction are two of the most fundamental tasks in signal processing and data analysis. In many scenarios, such as sensor networks and IoT, the underlying data is distributed among a collection of agents that cooperate to jointly solve the problem, \textit{i.e.}, find a consensus solution without sharing the data or sending it to a central unit \cite{Saber}. This paper addresses the two problems above mentioned  in a distributed setting, proposing and analysing two  algorithms that are instances of the \textit{distributed Banach-Picard 
iteration} (DBPI), which we have recently introduced and proved to enjoy local linear convergence \cite{FranciscodistributedPicardTheory}. In this introductory section, after presenting the formulations and motivations of the two problems considered, we review the DBPI and summarize our contributions and the tools that are used for the convergence proofs. 

\subsection{Problem Statement: Dimensionality Reduction} \label{PCA problem Statement}

Dimensionality reduction aims at representing high-dimensional data in a lower dimensional space, which can be crucial to reduce the computational complexity of manipulating  and processing this data, and is a core task in modern data analysis, machine learning, and related areas. The standard linear dimensionality reduction tool is \emph{principal component analysis} (PCA), which  allows expressing a high-dimensional dataset on the basis formed by the top eigenvectors of its sample covariance matrix. PCA  first appeared in the statistics community in the beginning of the 20th century \cite{pearson1901liii} and became one of the workhorses of statistical data analysis, with \emph{dimensionality reduction} being a notable application. Nowadays, with the ever increasing collection of data by spatially dispersed agents, developing algorithms for  distributed PCA constitutes a relevant area of research -- see, \textit{e.g.}, \cite{qu2002principal,liang2014improved,kannan2014principal,boutsidis2016optimal,garber2017communication,bai2005principal,kargupta2001distributed,qi2004global,abu2002distributed} (master-slave communication architecture), and \cite{scaglione2008decentralized,le2008distributed,yildiz2009distributed,suleiman2016performance,korada2011gossip,li2011distributed,schizas2015distributed,wu2017power} (arbitrarily meshed network communication architecture). For a recent and comprehensive review on these works, see, \textit{e.g.}, \cite{wu2018review}; for a very recent work see \cite{Gang}. 

In the (arbitrarily meshed network) distributed setting, consider  a set of $N$ agents linked by an undirected and connected communication graph; the nodes are the agents, and the edges represent the communication channels between the agents. Each agent $n$ holds a finite set  of points in $\mathbb{R}^d$, $\mathbf{Y}_n \subseteq \mathbb{R}^d$, and the agents seek to collectively find the top $m$ eigenvectors\footnote{This means the $m$ eigenvectors associated to the largest $m$ eigenvalues.} of
\begin{align*}
C=\frac{1}{M}\sum_{n=1}^N C_n,
\end{align*}
(assumed to be positive definite, \emph{i.e.,} $C \succ 0$), where  $M=\sum_{n=1}^N |\mathbf{Y}_n|$, \emph{i.e.}, the sum of the cardinalities of each $\mathbf{Y}_n$, and
\begin{align*}
C_n=\sum_{y \in \mathbf{Y}_n} yy^T.
\end{align*}

\subsection{Problem Statement: Distributed Parameter Estimation with Noisy and Faulty Measurements} \label{EM Problem Statement}
Consider  a collection of spatially distributed sensors monitoring the environment, a common scenario for information processing or decision making tasks see, \textit{e.g.,}  \cite{dimakis2010gossip,xiao2006distributed,boyd2006randomized,barbarossa2007decentralized,zhao2007information,schizas2007consensus,stankovic2010decentralized,sayed2014diffusion}. Often, these sensors communicate wirelessly, maybe in a harsh environment, which may result in faulty communications or sensor malfunctions \cite{pereira2018parameter}. The setup is modelled as follows: $N$ agents, linked by an undirected and connected communication graph, each holding an independent observation given by
\begin{align}
Y_n=Z_nh_n^T\mu^\star+W_n, \quad n=1,\ldots,N. \label{NoisyMeasurements}
\end{align}
In \eqref{NoisyMeasurements}, $\mu^\star \in \mathbb{R}^d$ is a fixed and unknown parameter, each $h_n\in \mathbb{R}^d$ is assumed to be known only at agent $n$, $\lbrace W_j\rbrace_{n=1}^N$ are independent and identically distributed (i.i.d.) zero-mean Gaussian random variables with variance $(\sigma^\star)^2$, and $\lbrace Z_n\rbrace_{n=1}^N$ are i.i.d. Bernoulli random variables ($Z_n\in\{0,1\}$) with $f_{Z_n}(z_n|p^\star)=(p^\star)^{z_n}(1-p^\star)^{1-z_n}$. This formulation models a scenario where sensor $n$ measures the parameter $\mu^\star$ with probability $p^\star$ and, with probability $1-p^\star$, it senses only noise, indicating a transducer failure \cite{pereira2018parameter}. The agents seek to collectively estimate $\mu^\star$, treating $p^\star$ and $\sigma^\star$ as nuisance (or latent) parameters.
Observe that if the binary variables $Z_n$ were not random, but fixed and known, the model could be regarded as a (distributed) linear regression problem. However, the randomness introduces an extra layer of difficulty accounting for potential sensor failures. 

A decentralized algorithm, rather than one where each sensor sends its data to a central node, is potentially more robust to faulty wireless communications that may render a sensor useless. Moreover, a decentralized algorithm can yield considerable energy savings \cite{dimakis2010gossip},  a very desirable feature.

\subsection{Distributed Banach-Picard Iteration (DBPI)}
Our recent  work \cite{FranciscodistributedPicardTheory} addressed a general distributed setup where $N$ agents, linked by a  communication network, collaborate to collectively find an attractor $x^\star$ of a map $H$ that can be implicitly represented as an average of local maps, \textit{i.e.}, 
\[
H=\frac{1}{N}\sum_{n=1}^N H_n,
\] 
where $H_n$ is the map held by agent $n$. As defined in \cite{FranciscodistributedPicardTheory}, an attractor  $x^\star$ of $H$ is  a fixed point thereof, $H(x^\star)=x^\star$, satisfying 
\begin{align} \label{JacobianCondition}
\rho\big(\mathbf{J}_H(x^\star)\big)<1,
\end{align}
where $\rho\big(\mathbf{J}_H(x^\star)\big)$ is the spectral radius of the Jacobian of $H$ at $x^\star$. Moreover, the map $H$ is not assumed to have a symmetric Jacobian and no global structural assumptions (\textit{e.g.}, Lipschitzianity or coercivity) are made.

The main contributions of \cite{FranciscodistributedPicardTheory} are a distributed algorithm to find $x^\star$ -- DBPI -- and the proof that it enjoys the \textit{local  linear convergence} of its centralized counterpart, \emph{i.e.}, of the (standard) Banach-Picard iteration: $x^{k+1}=H(x^k)$.

\subsection{Contributions and Related Work}
In this work, we propose addressing the distributed inference problems described in Subsections \ref{PCA problem Statement} and \ref{EM Problem Statement} using two instantiations of DBPI. More concretely, we propose:

\vspace{0.22cm}
\begin{enumerate}\itemsep0.25cm
	\item A distributed algorithm for PCA, which results from considering a map that can be implicitly written as an average of local maps and that has as a fixed point the solution to the PCA problem. 
	\item An algorithm that stems from formulating the problem described in Subsection \ref{EM Problem Statement} as a fixed point of a map induced by the stationary equations of the corresponding \textit{maximum likelihood estimation} (MLE) criterion. This map corresponds to the iterations of a slightly modified EM algorithm for a \emph{mixture of linear regressions} \cite{faria2010fitting}.
\end{enumerate}

\vspace{0.22cm}

The guarantees of local linear convergence for these distributed algorithms involve verifying condition \eqref{JacobianCondition} for the maps inducing them, which allows invoking the results from \cite{FranciscodistributedPicardTheory}. Consequently, a great portion of this paper is devoted to proving that \eqref{JacobianCondition} holds for these maps, which is far from trivial.

The distributed PCA problem (see \cite{dimakis2010gossip} for a review of distributed PCA) described in Subsection \ref{PCA problem Statement} was addressed in  \cite{gang2019fast} , where an algorithm termed \emph{accelerated distributed Sanger's algorithm} (ADSA) was proposed. The  authors consider a ``mini-batch variant'' of \textit{Sanger's algorithm} (SA, see \cite{sanger1989optimal}) and, inspired by \cite{shi2015extra}, arrive at ADSA. Although  no proof of convergence was presented in \cite{gang2019fast}, very recent work by the same authors proves convergence of their algorithm \cite{Gang2021}. Our contributions in this context are twofold: we show that ADSA is recovered by applying DBPI to SA, and that condition \eqref{JacobianCondition} holds for SA, thus, the guarantees of local convergence follow directly as a consequence of the results in \cite{FranciscodistributedPicardTheory}. We mention that no computer simulations of ADSA are presented in this work, since these can be found in \cite{gang2019fast}.

The problem presented in Subsection \ref{EM Problem Statement} was addressed in \cite{pereira2018parameter}, where \eqref{NoisyMeasurements} is regarded as a \emph{finite mixture model} \cite{mixtures_book}. To estimate $\mu^\star$, the authors proposed a distributed version of the \textit{expectation-maximization} (EM \cite{Dempster}) algorithm, termed  \textit{diffusion-averaging distributed EM}  (DA-DEM). However, DA-DEM, very much in the spirit of \cite{kar2008distributed,nedic2009distributed}, uses a diminishing step-size to achieve convergence,  leading to a sublinear convergence rate. Our contribution is an algorithm for this  problem that extends a slightly modified version of the centralized EM algorithm to distributed settings. The key challenge is to show that we can ``expect'' condition \eqref{JacobianCondition} to hold, and we dedicate a considerable amount of effort to this endeavor. We use the term ``expect'', since the operator underlying DBPI depends on the observed samples and, therefore, the existence of an attractor is a probabilistic question. Finally, we compare our algorithm with DA-DEM  through  Monte Carlo simulations, confirming the linear convergence rate of our algorithm and the sublinear convergence of DA-DEM.

There is considerable work on the ``probabilistic linear convergence'' of  EM   \cite{redner1984mixture}, \cite{sundberg1974maximum},  \cite{balakrishnan2017statistical}.
However,  neither the results in \cite{redner1984mixture}, nor those in \cite{sundberg1974maximum} encompass the mixture model underlying \eqref{NoisyMeasurements}. The mixture of regressions presented in \cite{balakrishnan2017statistical} bears some similarity with the model underlying \eqref{NoisyMeasurements}, but it is not the same: in  \cite{balakrishnan2017statistical}, $p$ is fixed at $1/2$ and $Z_n\in\{-1,1\}$ (rather than $\{0,1\}$), thus there are no measurements that are just noise. Furthermore, \cite{balakrishnan2017statistical} is primarily concerned with statistical guarantees for the error with respect to the ground truth, while we address the goal of establishing \eqref{JacobianCondition}.

As mentioned in \cite{pereira2018parameter}, there are two other relevant works on distributed EM, namely, \cite{nowak2003distributed} and \cite{forero2011distributed}. However (see \cite{pereira2018parameter}), both these works address a different problem of Gaussian mixture density estimation. Moreover, in the case of \cite{nowak2003distributed}, the algorithm demands a cyclic network topology, and, in \cite{forero2011distributed} the algorithm requires higher computational load on each node, since it is based on \emph{alternating direction method of multipliers}.

To summarize, we show that ADSA \cite{gang2019fast} is an instance of DBPI and  propose an algorithm to solve the mixture model underlying \eqref{NoisyMeasurements}, also as an instance of DBPI. Consequently, their corresponding guarantees of local linear convergence result from the attractor condition \eqref{JacobianCondition} for the map underlying the corresponding centralized counterparts. We compare DA-DEM and our proposed algorithm through numerical Monte Carlo simulations, and the  results confirm the linear convergence of our algorithm and the sublinear convergence of DA-DEM. 

\subsection{Organization of the Paper}
Section \ref{ReviewOfDistPic} briefly reviews the DBPI proposed in \cite{FranciscodistributedPicardTheory} and  the main convergence result therein proved. The  characterization of the fixed points of the ``mini-batch'' variant of Sanger's algorithm, as well as the attractor condition, are presented in Section \ref{PCA Section}. Section \ref{EM Section} describes the centralized variant of EM  underlying the proposed distributed algorithm for the problem in Subsection \ref{EM Problem Statement}, presents the verification of the attractor condition, and reports  the results of  simulations comparing DBPI with DA-DEM. 


\subsection{Notation}\label{Notation}
The set of real $n$ dimensional vectors with positive components is denoted by $\mathbb{R}^n_{>0}$.
Matrices and vectors are denoted by upper and lower case letters, respectively.  The spectral radius of a matrix $A$ is denoted by $\rho(A)$ and its Frobenius norm by $\|A\|_F$. Given a map $H$, $\mathbf{J}_H(x)$, and $dH(x)$ denote, respectively, the Jacobian of $H$  at $x$ and the differential of $H$ at $x$. Given a vector $v$, $v_s$ denotes its $s$th component; given a matrix $A$, $A_{st}$ denotes the element on the $s$th row and $t$th column and $A^T$  its transpose. 
The $d$-dimensional identity matrix is denoted by $I_d$, $\mathbf{1}_d$ is the $d$-dimensional vector of ones, and $\mathbf{0}_{m,n}$ is the $m \times n$ matrix of zeros.  Whenever convenient, we will denote a vector with two stacked blocks, $[v^T, u^T]^T$, simply as $(v,u)$. Given a square matrix $A$, $\mathcal{U}(A)$ is an upper triangular matrix of the same dimension as $A$ and whose upper triangular part coincides with that of $A$. Given a norm $\|\cdot\|$, $\bar{B}^{\|\cdot\|}_{\delta,\theta}$ denotes the closed ball of center $\theta$ and radius $\delta$ with respect to $\|\cdot\|$.  Random variables and vectors are  denoted by upper case letters and, for random variable $Y$, the probability density (or mass) function of $Y$ is denoted by $f_Y$.  The probability density of a Gaussian of mean $\mu$ and variance $\sigma^2$ is denoted by $\mathcal{N}(\cdot|\mu,\sigma^2)$. 

\section{Review of Distributed Banach-Picard Iteration} \label{ReviewOfDistPic}
Consider a network of $N$ agents, where the interconnection structure is represented by an undirected connected graph: the nodes correspond to the agents and an edge between two agents indicates  they can communicate (are neighbours). In the scenario considered in \cite{FranciscodistributedPicardTheory}, each agent $n \in \{1,...,N\}$ holds an operator $H_n:\mathbb{R}^d \to \mathbb{R}^d$, and the goal is to compute a fixed point of the average operator 
\begin{align}
H=\frac{1}{N}\sum_{n=1}^N H_n. \label{DistributedMap}
\end{align}
Each agent $n$ is restricted to performing computations involving $H_n$ and communicating with its neighbours. 

Our only assumption about $H$ is the existence of a locally attractive fixed point  $x^\star$, \emph{i.e.}, satisfying \eqref{JacobianCondition}.

Let $R$ be the map on $\mathbb{R}^{dN}$ defined, for $z=[z_1^T,\ldots,z_N^T]^T$ (with $z_j \in \mathbb{R}^d$ held by agent $j$) by 
\begin{align}\label{Residuals}
R(z)=\Big[ \big(H_1(z_1)-z_1\big)^T,\ldots,\big(H_N(z_N)-z_N\big)^T\Big]^T,
\end{align}
 and let $W=\tilde{W}\otimes I_d$, where $\tilde{W}$ is the so-called Metropolis weight matrix associated to the communication graph  \cite{xiao2004fast}. The algorithm proposed in \cite{FranciscodistributedPicardTheory} is presented in Algorithm 1, where $\alpha \in \mathbb{R}_{>0}$.

\begin{algorithm}
\caption{Distributed Banach-Picard Iteration (DBPI)}
\begin{algorithmic}[1]
  \small
  \STATE Initialization:
  \begin{equation*}
\begin{aligned}
z^{0} &\in \mathbb{R}^{dN},\\
z^1 &= Wz^0+\alpha R(z^0),\\
\end{aligned}
\end{equation*}
\STATE Update:
\begin{equation*}
z^{k+2}=(I+W)z^{k+1}-\frac{I+W}{2}z^k+\alpha\big(R(z^{k+1})-R(z^k)\big).
\end{equation*} 
\end{algorithmic}
\end{algorithm}

Informally,  in \cite{FranciscodistributedPicardTheory}, we show that $\alpha$ can be chosen such that if $z^k$ gets sufficiently close to $\mathbf{1}\otimes x^\star$, then it converges to  $\mathbf{1}\otimes x^\star$ at least linearly (the precise statement and proof can be found in  \cite{FranciscodistributedPicardTheory}). Notice that $z$ being equal to $\mathbf{1}\otimes x^\star$ means that all agents are in consensus, holding a copy of the fixed point $x^\star$.

\section{Distributed PCA}  \label{PCA Section}

\subsection{Algorithm}
We obtain a distributed algorithm for solving the PCA problem described in section \ref{PCA problem Statement} as an instantiation  of DBPI  by  introducing a map $H$ with a fixed point at the desired solution. Moreover, the guarantees of local linear convergence follow as a result of verifying \eqref{JacobianCondition}.

The ``mini batch variant'' of Sanger's algorithm (SA) proposed in \cite{gang2019fast} and inspired by \cite{sanger1989optimal} is the Banach-Picard iteration 
\begin{align*}
X^{k+1}=H(X^k), 
\end{align*}
where $H:\mathbb{R}^{d\times m} \to \mathbb{R}^{d\times m}$ is given by
\begin{align} \label{SangerMap}
H(X)=X+\eta\Big(CX-X\mathcal{U}\big(X^TCX\big)\Big),
\end{align}
and $\mathcal{U}$ was defined in Subsection \ref{Notation}. Observe that $H$ can be written as an average of local maps, \emph{i.e.},
\begin{align*}
H=\frac{1}{N}\sum_{n=1}^N H_n,
\end{align*}
where $H_n:\mathbb{R}^{d\times m} \to \mathbb{R}^{d\times m}$ is defined by
\begin{align} \label{SangerDirection}
H_n(X)=X+\eta \Big(\frac{N}{M}C_nX-X\mathcal{U}\big(X^T\frac{N}{M}C_nX\big)\Big).
\end{align}

Let $R$ be the map on $\mathbb{R}^{(d\times m)N}$ defined, for $z=[z_1^T,\ldots,z_N^T]^T$, as in \eqref{Residuals}, with $H_n$ as in \eqref{SangerDirection}. The distributed algorithm presented in \cite{gang2019fast}, named ADSA,  is exactly the DBPI, \emph{i.e.}, Algorithm 1, with this choice of $R$.

\subsection{Convergence: Main Results}
The convergence analysis amounts to verifying the attractor condition \eqref{JacobianCondition} for $H$, thus establishing, as a corollary of the results in \cite{FranciscodistributedPicardTheory}, the local linear convergence of Algorithm 1 with each $H_n$ in \eqref{Residuals} defined as in \eqref{SangerDirection} (equivalently, ADSA). 

We start with the following lemma (proved in Appendix \ref{AppendixProofOfLemma}) showing that the solution sought in the PCA problem is a fixed point of $H$ (as defined in \eqref{SangerMap}).

\begin{lemma} \label{SangerFixedPointTheorem}
Let $C\succ 0$. If $X^\star \in \mathbb{R}^{d\times m}$ satisfies
\begin{align}\label{SangerFixedPoint}
CX^\star=X^\star\mathcal{U}((X^\star)^TCX^\star)
\end{align}
then, each column of $X^\star$ is either $0$ or a unit-norm eigenvector of $C$. Moreover, the columns are orthogonal, \emph{i.e.}, $(X^\star)^TX^\star$ is diagonal with the diagonal elements being either one or zero.
\end{lemma}

The following theorem guarantees that the Banach-Picard iteration of $H$ has local linear convergence to its fixed points. 

\

\begin{theorem} \label{SAngerJacobianTheorem}
Let $\lambda_1 > \ldots > \lambda_m >\lambda_{m+1}\geq \ldots\geq \lambda_d>0$ be the eigenvalues of $C$. Suppose that $X^\star$ is a $d\times m$ matrix such $Cx^\star_i=\lambda_i x^\star_i$ (where $x_i^\star$ denotes the $i$th column of $X^\star$ and $C$ is as defined in \eqref{SangerMap}), for $i=1, \ldots, m$, and $(X^\star)^TX^\star=I_m$. Then,  there exists $\eta^\star$ such that, for $0<\eta<\eta^\star$,
\begin{align*}
\rho\big(\mathbf{J}_H(X^\star)\big)<1.
\end{align*}
\end{theorem}

\

\begin{rem}
The invertibility of $C$ that is assumed in the statements of Lemma 1 and Theorem 1 is not a big restriction. In fact, if $C\succeq 0$ rather than $C\succ 0$, then $\tilde{C}=C+\epsilon I$ satisfies $\tilde{C}\succ 0$ and has the same eigenvectors as $C$.
\end{rem}

\subsection{Proof of  Theorem \ref{SAngerJacobianTheorem}}

First, note that $H(X)=I+\eta S(X)$, where
\begin{align*}
S(X)=CX-\mathcal{U}\big(X^TCX\big).
\end{align*}
This implies $\mathbf{J}_H(X^\star)=I+\eta\mathbf{J}_S(X^\star)$, and, as a consequence, each eigenvalue of $\mathbf{J}_H(X^\star)$ is of the form $1+\eta\beta$, with $\beta$ being an eigenvalue of $\mathbf{J}_S(X^\star)$. The idea is to show that these eigenvalues $\beta$ of $\mathbf{J}_S(X^\star)$ enjoy a key property: they are real-valued and negative, $\beta<0$. Such property means that, for sufficiently small $\eta>0$, we have $|1+\eta\beta|<1$. To establish this key property, we divide the proof of Theorem 1  in two lemmas: Lemma 2 and Lemma 3. 

Lemma 2 will show that the eigenvalues of $\mathbf{J}_S(X^\star)$ coincide with those of the  linear map from $\mathbb{R}^{d\times m}$ to $\mathbb{R}^{d\times m}$ given by
\begin{align}\label{DSunderChangeOfVariables1}
W \to \hat{D}W-WD-A\mathcal{U}(DA^TW+W^TAD),
\end{align}
where 
\begin{align}\label{DiagonalEigenvavluesMatrix}
D=\mbox{diag}(\lambda_1, \ldots, \lambda_m),
\end{align}
\begin{align}\label{ExtendedDiagonalEigenvavluesMatrix}
\hat{D}=(\lambda_1,\ldots,\lambda_m,\lambda_{m+1},\ldots,\lambda_d),
\end{align}
and
\begin{align}\label{AuxiliaryMatrixA}
A=\begin{bmatrix}
I_m\\
\mathbf{0}_{d-m,m}
\end{bmatrix}.
\end{align}

Lemma 3 will show that the eigenvalues of \eqref{DSunderChangeOfVariables1} are real and negative.

\begin{lemma} Let $X^\star$ satisfy the conditions of Theorem \ref{SAngerJacobianTheorem}. The eigenvalues of $\mathbf{J}_S(X^\star)$, where 
\begin{align*}
S:\mathbb{R}^{d\times m} &\to \mathbb{R}^{d\times m}\\
X &\to CX-\mathcal{U}(X^TCX)
\end{align*}
coincide with those of the linear map given by
\begin{align*}
W \to \hat{D}W-WD-A\mathcal{U}(DA^TW+W^TAD),
\end{align*}
with $D$, $\hat{D}$, and $A$ given by, respectively,  \eqref{DiagonalEigenvavluesMatrix}, \eqref{ExtendedDiagonalEigenvavluesMatrix}, and \eqref{AuxiliaryMatrixA}.
\end{lemma}

\begin{proof}
From the rules of matrix differential calculus (see \cite{magnus2019matrix} and \cite{lutkepohl1997handbook}), the differential of $S$ at $X$, denoted by $dS(X)$, is the linear map
\begin{align}\label{DifferentialOFSatX}
dX\to CdX-(dX)\mathcal{U}(X^TCX)-Xd\big(\mathcal{U}(X^TCX)\big)(X).
\end{align}
Observe that $\mathcal{U}$ is a linear map, thus the composition rule for differentials further yields
\begin{align}\label{CompositionRuleForDifferntials}
d\big(\mathcal{U}(X^TCX)\big)(X)= \mathcal{U}\big((dX)^TCX+X^TCdX\big).
\end{align}

By assumption, $CX^\star=X^\star D$ with $D$ given by \eqref{DiagonalEigenvavluesMatrix}, thus combining this with \eqref{DifferentialOFSatX} and \eqref{CompositionRuleForDifferntials}, $dS(X^\star)$, which we denote by $\hat{S}$ to simplify the notation, is given by
\begin{align*}
\hat{S}(dX)&=CdX-(dX)D\\
&\;\;\;\; -X^\star\mathcal{U}\big((dX)^TX^\star D+D(X^\star)^TdX\big).
\end{align*}
The  eigenvalues of $\mathbf{J}_S(X^\star)$ coincide with those of $dS(X^\star)=\hat{S}$, under the identification between Jacobians and differentials (see \cite{magnus2019matrix}); hence, we will study the eigenvalues of the latter.

 Let  $\hat{X}^\star$ be an extension of $X^\star$ to an orthonormal basis of eigenvectors of $C$, \emph{i.e.}, $(\hat{X}^\star)^T\hat{X}^\star=I_d$ and $C\hat{X}^\star=\hat{X}^\star \hat{D}$, where $\hat{D}$ is given by \eqref{ExtendedDiagonalEigenvavluesMatrix}. 
 To understand the eigenvalues of $\hat{S}$, consider the linear map given by  
\begin{align}\label{DifferentialChangeOfVariables}
V(dX) \to \hat{X}^\star dX
\end{align}
and observe that $V$ is an invertible linear map (in fact $V^{-1}(dX)=(\hat{X}^\star)^TdX$).
Eigenvalues are invariant under a similarity transformation, hence, the eigenvalues of $\hat{S}$ coincide with those of $V^{-1}\circ \hat{S}\circ V$ which, after renaming $dX$ by $W$, just amounts to the linear map
\begin{align*}
W \to \hat{D}W-WD-A\mathcal{U}(DA^TW+W^TAD),
\end{align*} with  $D$, $\hat{D}$, and $A$ given by, respectively, \eqref{DiagonalEigenvavluesMatrix}, \eqref{ExtendedDiagonalEigenvavluesMatrix}, and \eqref{AuxiliaryMatrixA}.
\end{proof}

In the proof of the following lemma it is crucial that the eigenvalues are in decreasing order.
\begin{lemma}
Let $D$, $\hat{D}$, and $A$ be defined, respectively,  by \eqref{DiagonalEigenvavluesMatrix}, \eqref{ExtendedDiagonalEigenvavluesMatrix}, and \eqref{AuxiliaryMatrixA}. The eigenvalues of the linear map from $\mathbb{R}^{d\times m}$ to $\mathbb{R}^{d\times m}$ defined by
\begin{align}\label{DSunderChangeOfVariables2}
W \to \hat{D}W-WD-A\mathcal{U}(DA^TW+W^TAD)
\end{align}
are real and negative.
\end{lemma}
\begin{proof}
Let $Z$ be an eigenvector (note that $Z$ is in fact a matrix) of \eqref{DSunderChangeOfVariables2} associated to the eigenvalue $\beta$, \emph{i.e.},
\begin{align}\label{EigenvEquationInVecForm}
\hat{D}Z-ZD-A\mathcal{U}(DA^TZ+Z^TAD)=\beta Z;
\end{align}
next, we  show that $\beta<0$.

Consider a block partition of $Z$ of the form
\begin{align*}
Z=\begin{bmatrix}
\tilde{Z}\\
\bar{Z}
\end{bmatrix},
\end{align*}
where $\tilde{Z}$ and  $\bar{Z}$ are, respectively, $m\times m$  and $(d-m)\! \times\! m$ matrices. The eigenvalue matrix equation \eqref{EigenvEquationInVecForm} induces the following system of matrix equations
\begin{align} \label{Two Eigenvalue equations1}
D\tilde{Z}-\tilde{Z}D-\mathcal{U}(D\tilde{Z}+\tilde{Z}^TD)&=\beta \tilde{Z},\\ \label{Two Eigenvalue equations2}
\bar{D}\bar{Z}-\bar{Z}D&=\beta \bar{Z},
\end{align}
where $\bar{D}=\mbox{diag}(\lambda_{m+1}, \ldots, \lambda_d)$.

There are two non-mutually-exclusive cases to consider: $\tilde{Z} \neq 0$ or $\bar{Z}\neq 0$ ($Z\neq 0$, by virtue of being an eigenvector).

\paragraph*{Case 1}Suppose  that $\bar{Z}_{st}\neq 0$. Then, \eqref{Two Eigenvalue equations2} implies that
\begin{align*}
\lambda_{m+s}\bar{Z}_{st}-\lambda_t\bar{Z}_{st}=\beta \bar{Z}_{st},
\end{align*}
and, hence, $\beta=\lambda_{m+s}-\lambda_t<0$.
\paragraph*{Case 2} Suppose that $\tilde{Z}_{st}\neq 0$. This case splits in two: either $s>t$ or $s \leq t$. If $s>t$, then \eqref{Two Eigenvalue equations1} and the ``upper triangularization'' operation yields
\begin{align} \label{Case 2 a)}
\lambda_s \tilde{Z}_{st}-\lambda_t\tilde{Z}_{st}=\beta \tilde{Z}_{st},
\end{align}
which, after dividing by $\tilde{Z}_{st}$, yields $\lambda_s-\lambda_t=\beta<0$. If $s \leq t$, then, 
\begin{align*}
\beta\tilde{Z}_{st}&=\lambda_s\tilde{Z}_{st}-\lambda_t\tilde{Z}_{st}-\mathcal{U}(D\tilde{Z}+\tilde{Z}^TD)_{st}\\
&=\lambda_s\tilde{Z}_{st}-\lambda_t\tilde{Z}_{st}-\lambda_s\tilde{Z}_{st}-\lambda_t\tilde{Z}_{ts}\\
&=-\lambda_t(\tilde{Z}_{st}+\tilde{Z}_{ts}).
\end{align*}
Next, notice that if $s<t$, then
$\tilde{Z}_{ts}$ can be assumed to be $0$, since, otherwise, we could deal with it as in \eqref{Case 2 a)} with the roles of $s$ and $t$ reversed to conclude $\beta<0$. Hence, assuming $\tilde{Z}_{ts}=0$, we obtain, after division by $\tilde{Z}_{st}$, that $\beta=-\lambda_t<0.$ Finally, if $s=t$, then $\beta=-2\lambda_t<0.$  
\end{proof}
\color{black}
\section{Parameter estimation with noisy measurements} \label{EM Section}
\subsection{Roadmap}

This is a rather long section, hence the need  for a road map. The analysis of \eqref{NoisyMeasurements} is simplified if the measurements are identically distributed besides being just independent and, therefore, we  start by introducing a probability distribution on  the vectors $h_n$ and a joint model on $(Y,H)$.

To estimate $\mu^\star$ ($p^\star$ and $\sigma^\star$ are treated as nuisance parameters), we  consider the stationary equations imposed by equating to zero the gradient of the log-likelihood function, a necessary condition satisfied by the maximum likelihood estimator (MLE). Once the particular form of the stationary equations is realized, we reformulate them as a fixed point equation of the form $g_1 \circ g_2(\theta^\star)=\theta^\star$ that naturally suggests the Banach-Picard iteration 
\begin{align*}
\theta^{k+1}=g_1\circ g_2(\theta^k).
\end{align*}

Observing that the map $g_1\circ g_2$  cannot be written as an average of local maps,  we switch to the map $H=g_2 \circ g_1$, which can be implicitly written as an average of local maps. With this map, we arrive at a distributed algorithm by considering the map $R$ (see section \ref{ReviewOfDistPic}) arising from $H$ and appealing to Algorithm 1.

Finally, we observe that the existence of a fixed point of $H=g_2 \circ g_1$ satisfying \eqref{JacobianCondition} follows from the existence of a fixed point of $g_1\circ g_2$ satisfying \eqref{JacobianCondition}. The final part of the section is thus devoted to verifying \eqref{JacobianCondition} for the map $g_1 \circ g_2$, and to a numerical simulation comparing Algorithm 1 and DA-DEM from \cite{pereira2018parameter}.

\subsection{Identically distributed observations}

Let $\theta^\star=(\mu^\star,p^\star,(\sigma^\star)^2) \in \Theta=\mathbb{R}^d\times (0,1)\times (0,+\infty)$ be an unknown and fixed vector which we term the ground truth.

The agents' measurements are assumed to be independent (see \eqref{NoisyMeasurements}); however, they are not identically distributed, given the presence of the vectors $h_n$ in \eqref{NoisyMeasurements}. To address this issue, let $Z \in \lbrace 0,1\rbrace$, $H \in \mathbb{R}^d$, and $Y \in \mathbb{R}$ be, respectively, a binary random variable, a random vector, and a real random variable. Suppose  the joint density on $(Y,H,Z)$ factors as
\begin{equation}\label{JointModel}
\begin{aligned}
f_{Y,H,Z}\big(y,h,z|\theta^\star\big)=f_{H}(h)f_Z(z|p^\star)f_{Y|H,Z}\big(y|h,z,\mu^\star,(\sigma^\star)^2 \big),
\end{aligned}
\end{equation}
where
\begin{equation*}
\begin{aligned}
f_H(h)&=\mathcal{N}(h|0,I_d),\\
f_Z(z|p^\star)&=(p^\star)^z(1-p^\star)^{1-z},\\
\end{aligned}
\end{equation*}
and
\begin{equation}
\begin{aligned}
f_{Y|H,Z}\big(y|h,z,\theta^\star \big)=\mathcal{N}\big(y|h^T\mu^\star, (\sigma^\star)^2\big)^z\mathcal{N}\big(y|0,(\sigma^\star)^2\big)^{1-z}.
\end{aligned}
\end{equation}

Instead of assuming that the $h_n$ are fixed as in \cite{pereira2018parameter}, we assume that each sensor $n$ has a measurement $(y_n,h_n)$, where $(y_n,h_n,z_n)$ was drawn from \eqref{JointModel}, but agent $n$ has no knowledge of $z_n$. After marginalization, the joint density of $Y,H$ is given by
\begin{equation}
\begin{aligned}\label{JointOnYandH}
&f_{Y,H}\big(y,h|\theta^\star)=f_H(h)f_{Y|H}(y|h,\theta^\star)\\
=&f_H(h)\Big(p^\star\mathcal{N}\big(y|h^T\mu^\star,(\sigma^\star)^2\big)+(1-p^\star)\mathcal{N}\big(y|0,(\sigma^\star)^2\big)\Big),
\end{aligned} 
\end{equation}
which is a mixture model \cite{bishop2006pattern}.

To estimate $\mu^\star$, the agents seek   $\theta\in  \Theta$ such that
\begin{equation}\label{MLE}
\frac{1}{N}\sum_{n=1}^N \nabla_\theta \phi (y_n,h_n,\theta)=0,
\end{equation}
where  $\phi$ is the log-likelihood of $(Y,H)$, \emph{i.e.},
\begin{equation}
\begin{aligned}
\phi(y,h,\theta)&=\log(f_{Y,H}(y,h|\theta)\\
&=\log \big(f_H(h)\big)+\log\big(f_{Y|H}(y|h,\theta)\big).
\end{aligned}
\end{equation}	
Since $f_H(h)$ does not depend on $\theta$, $$\nabla_\theta \phi(y,h,\theta)=\nabla_\theta \log \big(f_{Y|H}(y|h,\theta)\big);$$ in other words, \eqref{MLE} is a necessary condition satisfied by the MLE corresponding to the log-likelihood function  $\log f_{Y|H}(y|h,\theta)$, thus independent of $f_H$.
\subsection{Gradient of $\phi$ and the centralized algorithm}
Before  explicitly writing  the stationary equations corresponding to \eqref{MLE}, we introduce the \emph{responsibility functions} \cite{bishop2006pattern}, 
\begin{equation} \label{ResponsabilityFunction}
r(y,h,\theta)=\frac{p\mathcal{N}(y|h^T\mu,\sigma^2)}{p\mathcal{N}(y|h^T\mu,\sigma^2)+(1-p)\mathcal{N}(y|0,\sigma^2)}.
\end{equation}
Notice that $r(y,h,\theta)=\mathbb{P}(z=1|y,h,\theta)$, the posterior probability that the observation $y$ was not a result of measuring only noise. 

For reasons that will be clear later, the following set of equalities, which can be easily verified, will be convenient:
\begin{equation} \label{GradientOfPhi}
\begin{aligned}
\sigma^2\nabla_\mu \phi(y,h,\theta)&=r(y,h,\theta)(y-h^T\mu)h\\
p+p(1-p)\frac{\partial \phi}{\partial p}(y,h,\theta)&=r(y,h,\theta)\\
\sigma^2+2(\sigma^2)^2\frac{\partial \phi}{\partial \sigma^2}(y,h,\theta)&=
r(y,h,\theta)(y-h^T\mu)^2\\
&+\big(1-r(y,h,\theta)\big)y^2.
\end{aligned}
\end{equation}
Using \eqref{GradientOfPhi},  \eqref{MLE} can be explicitly written as
\begin{align}\label{Stability for mu}
\Big(\frac{1}{N}\sum_{n=1}^N\Gamma(y_n,h_n,\theta)\Big)\mu &=\frac{1}{N}\sum_{n=1}^N \psi(y_n,h_n,\theta)\\
p&=\frac{1}{N}\sum_{n=1}^N r(y_n,h_n,\theta)\\\label{Stability for sigma}
\sigma^2&=\frac{1}{N}\sum_{n=1}^N \gamma(y_n,h_n,\theta),
\end{align}
where 
\begin{align*}
\Gamma(y,h,\theta)&=r(y,h,\theta)hh^T,\\
\psi(y,h,\theta)&=r(y,h,\theta)yh,\\
\gamma(y,h,\theta)&=r(y,h,\theta)(y-h^T\mu)^2+\big(1-r(y,h,\theta)\big)y^2.
\end{align*}
If the matrix $\frac{1}{N}\sum_{n=1}^N \Gamma(y_n,h_n,\theta)$ is invertible, then \eqref{Stability for mu}-\eqref{Stability for sigma} can be written as a fixed point equation.\footnote{The invertibility of this matrix is assumed throughout the rest of the paper. In fact, if $N$ is sufficiently large - greater than $d$ - this happens with probability one.} This constitutes the motivation for the centralized algorithm that we suggest next (see Algorithm 2) and from which we will derive the distributed version; observe that it is the Banach-Picard iteration motivated by \eqref{Stability for mu}-\eqref{Stability for sigma}.

Another way to write  \eqref{Modified EM1}-\eqref{Modified EM3} (see Algorithm 2 below) is $\theta^{k+1}=g_1\circ g_2(\theta^k)$, where
\begin{align*}
g_2(\theta)=\frac{1}{N}\Big(&\sum_{n=1}^N \Gamma(y_n,h_n,\theta),\sum_{n=1}^N \psi(y_n,h_n,\theta),\\
&\sum_{n=1}^N r(y_n,h_n,\theta),\sum_{n=1}^N \gamma(y_n,h_n,\theta)\Big)
\end{align*}
and
\begin{align*}
g_1(\Gamma,\psi,p,\sigma^2)=&\big(\Gamma^{-1}\psi,p,\sigma^2\big).
\end{align*}

\begin{algorithm}

\caption{Centralized variant of EM}
\begin{algorithmic}
\scriptsize
\STATE Initialization:
  \begin{equation*}
\begin{aligned}
\theta^0=\big(\mu^0,p^0, (\sigma^0 )^2\big)\in \Theta
\end{aligned}
\end{equation*}
    \STATE Update: $\theta^{k+1}=\big(\mu^{k+1},p^{k+1},(\sigma^{k+1})^2\big)$, where \begin{align}\label{Modified EM1}
\mu^{k+1}&=\Big(\frac{1}{N}\sum_{n=1}^N\Gamma(y_n,h_n,\theta^k)\Big)^{-1}\frac{1}{N}\sum_{n=1}^N \psi(y_n,h_n,\theta^k)\\\label{Modified EM2}
p^{k+1}&=\frac{1}{N}\sum_{n=1}^N r(y_n,h_n,\theta^k)\\\label{Modified EM3}
(\sigma^{k+1})^2&=\frac{1}{N}\sum_{n=1}^N \gamma(y_n,h_n,\theta^k).
\end{align} 
\end{algorithmic}
\end{algorithm}

\subsection{Distributed Algorithm}
Although, $g_2$ is an average of local maps, the map $g_1\circ g_2$ is not, due to the  matrix inversion in \eqref{Modified EM1}. As a consequence, \eqref{Modified EM1}-\eqref{Modified EM3} cannot be directly extended  to a distributed algorithm. However,  switching the order of $g_1$ and $g_2$ results in a map that can be implicitly written as an average of local maps. In fact, instead of the iteration $\theta^{k+1}=g_1\circ g_2(\theta^k)$,   consider the iteration
\begin{align*}
z^{k+1}=H(z^k),
\end{align*}
where $H=g_2\circ g_1$ and $z=(\Gamma,\psi,p,\sigma^2)$.

Let
\begin{align*}
G_n(\theta)=\Big(&\Gamma\big(y_n,h_n,\theta\big),\psi\big(y_n,h_n,\theta\big),\\
&r\big(y_n,h_n,\theta\big),\gamma\big(y_n,h_n,\theta\big)\Big),
\end{align*}
and it follows that $H=\frac{1}{N}\sum_{n=1}^N H_n$, where
\begin{align} \label{EachAgentEM}
H_n(z)=G_n\circ g_1(z).
\end{align} 

To conclude,  the distributed algorithm we suggest amounts to Algorithm 1, with $R: \mathbb{R}^{(d^2+d+2)N}\to \mathbb{R}^{(d^2+d+2)N}$ defined, for $z=(z_1^T,\ldots,z_N^T)^T$, as in \eqref{Residuals}, and $H_n$ as in \eqref{EachAgentEM}. Additionally, following \cite{pereira2018parameter}, we suggest the initialization
\begin{align}\label{DistributedEmInitialization}
z_n^0=\sum_{m=1}^N \tilde{W}_{nm} G_m\big(\frac{y_nh_n}{h_n^Th_n},\frac{1}{2},\frac{y_n^2}{2}\big).
\end{align}
Some remarks are due: 
\begin{description}
\item[a)] The existence of a fixed point of $g_1 \circ g_2$ satisfying \eqref{JacobianCondition} is addressed in section \ref{EM Proof of convergence};
\item[b)] The existence of a fixed point of $g_2 \circ g_1$ satisfying \eqref{JacobianCondition} follows from the existence of a fixed point of $g_1 \circ g_2$ satisfying \eqref{JacobianCondition}, by the chain rule;
\item[c)] Expanding \eqref{Modified EM3} yields
\begin{equation*}
\begin{aligned}
(\sigma^{k+1})^2&=
\frac{1}{N}\sum_{n=1}^N r(y_n,h_n,\theta^k)(y_n-h_n^T\mu^k)^2\\
&+\big(1-r(y_n,h_n,\theta^k)\big)y_n^2,
\end{aligned}
\end{equation*}
and, if the update rule is modified according to
\begin{equation*}
\begin{aligned}
(\sigma^{k+1})^2&=\frac{1}{N}\sum_{n=1}^N r(y_n,h_n,\theta^k)(y_n-h_n^T\mu^{k+1})^2\\
&+\big(1-r(y_n,h_n,\theta^k)\big)y_n^2,
\end{aligned}
\end{equation*}
then,  a straightforward manipulation recovers the  EM algorithm presented in \cite{pereira2018parameter}. Moreover, the  EM algorithm derived in \cite{pereira2018parameter} is still amenable to a distributed implementation using Algorithm 1. However, we found it easier to prove \eqref{JacobianCondition} for this variant of EM, than for the standard EM.
\end{description}

\subsection{Convergence Analysis} \label{EM Proof of convergence}
The proof of local linear convergence of the centralized variant of EM, \emph{i.e.}, Algorithm 2,  is not trivial. In fact, this question is probabilistic in nature, because  updates \eqref{Modified EM1}-\eqref{Modified EM3}  depend on observations that are, in turn, samples from a probability distribution. Before presenting the main convergence result (Theorem \ref{FundamentalEmTheorem}), we need to introduce some definitions and only  one mild assumption that is instrumental in the proof of Lemma 4 below: the Fisher information at $\theta^\star$, given by,
\begin{align}\label{FisherInformationCondition}
\mbox{I}(\theta^\star)=\mathbb{E}_{\theta^\star}\Big[\nabla_\theta \phi(y,h,\theta^\star)\big(\nabla_\theta \phi(y,h,\theta^\star)\big)^T\Big],
\end{align}
is non-singular.

Let $T_N=g_1\circ g_2$ denote the map underlying the Banach-Picard iteration \eqref{Modified EM1}-\eqref{Modified EM3}.\footnote{The subscript $N$ emphasizes that $T_N$ depends on $N$ observations.} A straightforward manipulation, using \eqref{GradientOfPhi}, shows that
\begin{align*}
T_N(\theta)=\theta + \big(A_N(\theta)\big)^{-1}\frac{1}{N}\sum_{n=1}^N\nabla_\theta \phi(y_n,h_n,\theta),
\end{align*}
where
\begin{align*}
A_N(\theta)=\begin{bmatrix}
\frac{1}{N}\sum_{n=1}^N \frac{1}{\sigma^2}\Gamma(y_n,h_n,\theta)& \mathbf{0}&\mathbf{0}\\
\mathbf{0}&\frac{1}{p(1-p)}&0\\
\mathbf{0}&0&\frac{1}{2(\sigma^2)^2}
\end{bmatrix}.
\end{align*}

Before stating the main result, we introduce the ``infinite sample'' map, \emph{i.e.},
\begin{align*}
T(\theta)=\theta+\big(A(\theta)\big)^{-1}\mathcal{L}(\theta),
\end{align*}
where
\begin{align*}
A(\theta)=\begin{bmatrix}
\mathbb{E}_{\theta^\star}\Big[\frac{1}{\sigma^2}\Gamma(y,h,\theta)\Big]& \mathbf{0}&\mathbf{0}\\
\mathbf{0}&\frac{1}{p(1-p)}&0\\
\mathbf{0}&0&\frac{1}{2(\sigma^2)^2}
\end{bmatrix},
\end{align*}
and  $$\mathcal{L}(\theta)=\mathbb{E}_{\theta^\star}\big[\nabla \phi(y,h,\theta)\big].$$

The next lemma shows that $T$ is a ``natural'' map to consider.
 
\begin{lemma}
The ``infinite sample map'', \emph{i.e.}, $T$, has the following properties
\begin{description}
\item[a)] For fixed $\theta$, $T_N(\theta)$ converges in probability to $T(\theta)$;
\item[b)]  The ground truth $\theta^\star$ is a fixed point of $T$, \emph{i.e.} $T(\theta^\star)=\theta^\star$;
\item[c)] The attractor condition \eqref{JacobianCondition} holds for $T$ at $\theta^\star$, \emph{i.e.},
\begin{align*}
\rho\big(\mathbf{J}_T(\theta^\star)\big)<1.
\end{align*}
\end{description}
\end{lemma}

\begin{proof}

The proof of b) amounts to a straightforward verification and the proof of a) follows by the weak law of large numbers, hence, we will focus on the proof of c) which relies on the \emph{principle of missing information} (see \cite{mclachlan2007algorithm}, page 101), and the assumption on the Fisher Information condition, \emph{i.e.},  \eqref{FisherInformationCondition}.

Under suitable regularity conditions (see appendix \ref{RegularityConditionSect}) that hold for the model \eqref{JointOnYandH}, $$\mathbb{E}_{\theta^\star}\Big[\nabla^2_\theta \phi(y,h,\theta^\star)\Big]=-\mbox{I}(\theta^\star).$$

Additionally, a simple calculation reveals that $A(\theta^\star)$ coincides with the Fisher information of the complete data model \eqref{JointModel}, \emph{i.e.,}$$A(\theta^\star)=\mbox{I}_c(\theta^\star).$$ The non-singularity assumption on $\mbox{I}(\theta^\star)$ (see \eqref{FisherInformationCondition}), together with the principle of missing information (see \cite{mclachlan2007algorithm}, page 101), implies that $$0\prec \mbox{I}(\theta^\star) \preceq \mbox{I}_c(\theta^\star).$$ All these observations show that
\begin{align} \label{VAriantEMPopulationEMContraction}
\mathbf{J}_T(\theta^\star)=I-\big(\mbox{I}_c(\theta^\star)\big)^{-1}\mbox{I}(\theta^\star),
\end{align}
and, Theorem 7.7.3 of \cite{horn2012matrix}, together with $0\prec \mbox{I}(\theta^\star) \preceq \mbox{I}_c(\theta^\star)$, implies that
\begin{align*}
\rho\big(\mathbf{J}_T(\theta^\star)\big)<1,
\end{align*}
concluding that $\theta^\star$ is an attractor of $T$.
\end{proof}
We recall that the goal of this section is to show that the probability that $T_N$ has an attractor  approaches $1$ as $N$ tends to infinity; the strategy is to derive this from the existence of an attractor of $T$, \emph{i.e.}, c) in Lemma 4. Pointwise convergence in probability, \emph{i.e.}, a) in Lemma 4, is  not enough to arrive at this result. In fact, the proof is built on a stronger notion that is a probabilistic version of \emph{uniform}, rather than pointwise, convergence of maps. This is the content of Theorem \ref{CrucialLemma} below.

Observe that if $\theta_N$ is a fixed point of $T_N$, then 
\begin{align*}
\mathbf{J}_{T_N}(\theta_N)=I+\big(A_N(\theta_N)\big)^{-1}\frac{1}{N}\sum_{n=1}^N \nabla_\theta^2\phi(y_n,h_n,\theta_N),
\end{align*}
so let
\begin{align*}
T_N^\prime(\theta)&=I+\big(A_N(\theta)\big)^{-1}\frac{1}{N}\sum_{n=1}^N \nabla_\theta^2\phi(y_n,h_n,\theta),\\
T^\prime(\theta)&=I+\big(A(\theta)\big)^{-1}\mathbb{E}_{\theta^\star}\Big[\nabla^2_\theta\phi(y,h,\theta)\Big].
\end{align*}

\begin{rem} Note that the maps $T_N^\prime(\theta)$ and $T^\prime(\theta)$ only coincide with the jacobian maps $\mathbf{J}_{T_N}(\theta)$ and $\mathbf{J}_T(\theta)$ at fixed points.
\end{rem}

The uniform convergence in probability is expressed in the next theorem, whose proof can be found in appendix \ref{ProofOfCrucialLemma}. For the statement, recall (see the notation section) that $\bar{B}^{\|\cdot\|}_{\delta,\theta^\star}$ is the closed ball of center $\theta^\star$ and radius $\delta$, with respect to the metric induced by the norm $\|\cdot\|$.

\begin{theorem}

\label{CrucialLemma} Let $\delta>0$ and $\|\cdot\|$ be any norm. With $T_N$, $T_N^\prime$, $T^\prime$, and $T$ as defined above, 
\begin{align}\label{CrucialTheoremFirstClaim}
&\sup_{\theta \in \bar{B}^{\|\cdot\|}_{\delta,\theta^\star}}\big\|T_N(\theta)-T(\theta)\big\|\to 0\\ \label{CrucialTheoremSecondClaim}
&\sup_{\theta \in \bar{B}^{\|\cdot\|}_{\delta,\theta^\star}}\big\|T_N^\prime(\theta)-T^\prime(\theta)\big\|\to 0
\end{align}
both in probability, as $N \to \infty$.
\end{theorem}

We  now state the main convergence result.

\

\begin{theorem}\label{FundamentalEmTheorem}
There exists $\delta>0$  and a norm $\|\cdot\|$ such that
\begin{align*}
\mathbb{P}_{\theta^\star}\Big(\sup_{\theta \in \bar{B}^{\|\cdot\|}_{\delta,\theta^\star}}\big\|T_N(\theta)-\theta^\star\|\leq \delta \Big)&\to 1 \quad \text{and}\\
\mathbb{P}_{\theta^\star}\Big(\sup_{\theta \in \bar{B}^{\|\cdot\|}_{\delta,\theta^\star}}\big\|T_N^\prime(\theta)\|< 1 \Big)&\to 1,
\end{align*}
where $\|T_N^\prime(\theta)\|$ is the induced matrix norm.\footnote{The measurability of the maps in this Theorem are a consequence of Proposition 7.32 in \cite{bertsekas2004stochastic}.}
\end{theorem}

\

Before presenting the proof, we explain why Theorem \ref{FundamentalEmTheorem} encapsulates the notion that, with probability approaching $1$, the map $T_N$ has an attractor. Let
\begin{align*}
\mathcal{A}_N&=\big\lbrace (\mathbf{y},\mathbf{h}) \in \mathbb{R}^N\times \mathbb{R}^{dN}: \sup_{\theta \in \bar{B}^{\|\cdot\|}_{\delta,\theta^\star}}\big\|T_N(\theta)-\theta^\star\|\leq \delta\big\rbrace\\
\mathcal{B}_N&=\big\lbrace (\mathbf{y},\mathbf{h}) \in \mathbb{R}^N\times \mathbb{R}^{dN}: \sup_{\theta \in \bar{B}^{\|\cdot\|}_{\delta,\theta^\star}}\big\|T_N^\prime(\theta)\|< 1\big\rbrace.
\end{align*}

\begin{rem} Informally, observe that the set $\mathcal{A}_N$ is the set of ``samples'' where the ball $\bar{B}_{\delta,\theta^\star}^{\|\cdot\|}$ is invariant under $T_N$, i.e, $$T_N\big(\bar{B}_{\delta,\theta^\star}^{\|\cdot\|}\big)\subseteq \bar{B}_{\delta,\theta^\star}^{\|\cdot\|},$$ and that the set $\mathcal{B}_N$ is, from remark 2, the set of ``samples''  where the Jacobian of $T_N$ satisfies \eqref{JacobianCondition} at a fixed points. By noting that a continuous map  from a convex compact space into itself has a fixed point (Brouwer's fixed point theorem), it follows that if $(\mathbf{y},\mathbf{h})$ is in $\mathcal{A}_N$, then $T_N$ has a fixed point. Moreover, if $(\mathbf{y},\mathbf{h})$ is in $\mathcal{A}_N \cap \mathcal{B}_N$ then $T_N$ has a fixed point satisfying \eqref{JacobianCondition}. All of this is made precise below.
\end{rem}

The statement of Theorem \ref{FundamentalEmTheorem} is that the (non-random sequences) $\mathbb{P}_{\theta^\star}(\mathcal{A}_N)$ and $\mathbb{P}_{\theta^\star}(\mathcal{B}_N)$ both tend to $1$. The  inequalities
\begin{align*}
\mathbb{P}_{\theta^\star}(\mathcal{A}_N)+\mathbb{P}_{\theta^\star}(\mathcal{B}_N)-1 &\leq \mathbb{P}_{\theta^\star}(\mathcal{A}_N \cap\mathcal{B}_N)\leq \mathbb{P}_{\theta^\star}(\mathcal{A}_N)
\end{align*}
imply that 
\begin{align*}
\mathbb{P}_{\theta^\star}(\mathcal{A}_N \cap\mathcal{B}_N)\to 1.
\end{align*}
Now note that, if both inequalities hold, namely

\begin{align}\label{BrouwersInequality}
&\sup_{\theta \in \bar{B}^{\|\cdot\|}_{\delta,\theta^\star}}\big\|T_N(\theta)-\theta^\star\|\leq \delta\\ \label{JacobianOFBrouewer}
&\sup_{\theta \in \bar{B}^{\|\cdot\|}_{\delta,\theta^\star}}\big\|T_N^\prime(\theta)\|< 1,
\end{align}
then \eqref{BrouwersInequality}, together with Brouwer's fixed point theorem (see \cite{borwein2010convex}, page 180) implies that $T_N$ has a fixed point $\theta_N$ in $\bar{B}^{\|\cdot\|}_{\delta,\theta^\star}$ (this idea is loosely inspired by \cite{van2000asymptotic}, page 69). Moreover, being a fixed point, at a $\theta_N$ it holds (see remark 2) that $T^\prime_N(\theta_N)= \mathbf{J}_{T_N}(\theta_N)$, so, \eqref{JacobianOFBrouewer} implies that \begin{align*}
\rho\big(\mathbf{J}_{T_N}(\theta_N)\big)\leq \big\|\mathbf{J}_{T_N}(\theta_N)\|\leq \sup_{\theta \in \bar{B}^{\|\cdot\|}_{\delta,\theta^\star}}\big\|T_N^\prime(\theta)\|< 1.
\end{align*}
This explains why Theorem \ref{FundamentalEmTheorem} expresses the notion that we can ``expect'' \eqref{JacobianCondition} to hold for $T_N$. In fact, from the above, the event
$$
\mathcal{C}_N=\big\lbrace (\mathbf{y},\mathbf{h}) : T_N \text{ has a fixed point satisfying \eqref{JacobianCondition}}\big\rbrace \\
$$ contains the event $\mathcal{A}_N \cap \mathcal{B}_N$, and the probability of this last event is approaching 1.
\subsubsection*{Proof of Theorem \ref{FundamentalEmTheorem}}
Let $\|\cdot\|$ be any norm. Then
\begin{align}\label{FirstInequalityOfProof}
\|T_N(\theta)-\theta^\star\| \leq \|T_N(\theta)-T(\theta)\|+\|T(\theta)-\theta^\star\|.
\end{align}
From Lemma 4 c),
\begin{align*}
\rho\big(\mathbf{J}_T(\theta^\star)\big)<1.
\end{align*}

From the proof of Ostrowski's Theorem (see  \cite{ortega2000iterative}, page 300), there exists a norm $\|\cdot\|$ on $\mathbb{R}^{d+2}$, an open neighborhood $\mathcal{V}$ of $\theta^\star$, and $\lambda<1$, such that
\begin{description}
\item[1)] $\|T(\theta)-\theta^\star\| \leq \lambda\|\theta-\theta^\star\|$, for $\theta \in \mathcal{V}$;
\item[2)] $\|\mathbf{J}_T(\theta^\star)\|<1$, where here the norm is the induced matrix norm.
\end{description}
Choose $\delta$ sufficiently small such that 
\begin{description}
\item[i)] $ \bar{B}^{\|\cdot\|}_{\delta,\theta^\star} \subseteq \mathcal{V}$;
\item[ii)] $\|T^\prime(\theta)\|=\|\mathbf{J}_T(\theta)\|\leq \beta<1$, for $\theta \in \bar{B}^{\|\cdot\|}_{\delta,\theta^\star}$,
\end{description}
where the validity of ii) follows from the compactness of $\bar{B}^{\|\cdot\|}_{\delta,\theta^\star}$  and the continuity of $T^\prime$.

Now, for any $\theta \in \bar{B}^{\|\cdot\|}_{\delta,\theta^\star}$, \eqref{FirstInequalityOfProof} implies that
\begin{align*}
\|T_N(\theta)-\theta^\star\| \leq \|T_N(\theta)-T(\theta)\|+\lambda \delta
\end{align*}
and, hence,
\begin{align}\label{Inequlity1}
\sup_{\theta \in \bar{B}^{\|\cdot\|}_{\delta,\theta^\star}}\big\|T_N(\theta)-\theta^\star\big\|\leq \sup_{\theta \in \bar{B}^{\|\cdot\|}_{\delta,\theta^\star}}\big\|T_N(\theta)-T(\theta)\big\|+\lambda\delta.
\end{align}
A similar reasoning shows that
\begin{align} \label{Inequlity2}
\sup_{\theta \in \bar{B}^{\|\cdot\|}_{\delta,\theta^\star}}\big\|T_N^\prime(\theta)\|\leq \sup_{\theta \in \bar{B}^{\|\cdot\|}_{\delta,\theta^\star}}\big\|T_N^\prime(\theta)-T^\prime(\theta)\big\|+\beta.
\end{align}

To conclude, we appeal to Theorem \ref{CrucialLemma}, and we show that it implies the result.  Let $\epsilon_1=(1-\lambda)\delta$ and $\epsilon_2=\frac{1-\beta}{2}$. From the properties of $\lambda$ and $\beta$, it holds that $\epsilon_1>0$ and $0<\epsilon_2<1$. By the definition of convergence in probability, it holds that
\begin{align*}
&\mathbb{P}_{\theta^\star}\Big(\sup_{\theta \in \bar{B}^{\|\cdot\|}_{\delta,\theta^\star}}\big\|T_N(\theta)-T(\theta)\|\leq \epsilon_1\Big)\to 1\\
&\mathbb{P}_{\theta^\star}\Big(\sup_{\theta \in \bar{B}^{\|\cdot\|}_{\delta,\theta^\star}}\big\|T_N^\prime(\theta)-T^\prime(\theta)\|\leq \epsilon_2\Big)\to 1.
\end{align*}
From \eqref{Inequlity1}, \eqref{Inequlity2}, and the forms of $\epsilon_1$ and $\epsilon_2$, we conclude the result.

\subsection{Simulations}
In this section, we compare our algorithm with the one from \cite{pereira2018parameter} (DA-DEM) through  Monte Carlo simulations. The parameters generated once and fixed throughout all Monte Carlo runs were:  $d=3$, $N=100$, a unit-norm vector  $\mu^\star \in \mathbb{R}^d$, $p^\star=0.7$, and an undirected connected graph on $N$ nodes with connectivity radius $r_c= 0.18$ \footnote{$N$ points were randomly deployed on the unit square; two points were then  connected by an edge if their distance was less than $r_c$.}.

\

Each Monte Carlo run consisted in 
\begin{description}
\item[1)] Generating a data set: each $h_n$ was independently sampled from a Gaussian with zero mean and covariance $I_3$; the variance of the noise $(\sigma^\star)^2$ was generated according to
\begin{align*}
(\sigma^\star)^2=\frac{\|\mathbf{H}\|_F^2}{N \times \mbox{SNR}},
\end{align*}
with $\mathbf{H}^T=[h_1\ldots h_N]$ and where $\text{SNR}$ is the signal to noise ratio (we experimented with $\text{SNR} \in \lbrace 10\text{dB}, 20\text{dB}\rbrace$). Finally, each $y_n$ was sampled according to $f_{Y|H}$ (see \eqref{JointOnYandH}), with $h_n$, $\mu^\star$,  $p^\star$, and $(\sigma^\star)^2$.
\item[2)] Computing $10000$ iterations of the algorithm proposed in  \cite{pereira2018parameter},
with $\rho \in \lbrace 2, 3, 4 \rbrace$, and  of Algorithm 1,
with $\alpha \in \lbrace 0.001, 0.005, 0.01\rbrace$. Both algorithms were initialized according to \eqref{DistributedEmInitialization}.
\end{description}
 The performance metrics consisted in finding a fixed point using the centralized operators as follows. 

We first computed 
\begin{align}
&\theta^0(\alpha)=\frac{1}{N}\sum_{n=1}^N g_1\big(z_n^{10000}(\alpha)\big)\label{InitializationForVariantCentEm}\\
&\theta^0(\rho)=\frac{1}{N}\sum_{n=1}^N \hat{g}_1\big(z_n^{10000}(\rho)\big) \label{InitializationForCentEm},
\end{align}
where: $\alpha \in \lbrace 0.001, 0.005, 0.01\rbrace$; $\rho \in \lbrace 2, 3, 4 \rbrace$; $\hat{g}_1$ corresponds to the map arising from the standard EM algorithm derived in \cite{pereira2018parameter}. In fact, as seen in \cite{pereira2018parameter}, the EM algorithm can be written as 
\begin{align*}
\theta^{k+1}=\hat{g}_1\circ \hat{g_2}(\theta^k),
\end{align*}
where 
\begin{align*}
\hat{g}_2(\theta)=\frac{1}{N}\Big(&\sum_{n=1}^N \Gamma(y_n,h_n,\theta),\sum_{n=1}^N \psi(y_n,h_n,\theta),\\
&\sum_{n=1}^N r(y_n,h_n,\theta),\sum_{n=1}^N y_n^2\Big)
\end{align*}
and
\begin{align*}
\hat{g}_1(\Gamma,\psi,p,a)=&\big(\Gamma^{-1}\psi,p,a-\psi^T\Gamma^{-1}\psi\big).
\end{align*}
We ran the algorithms, with initialization as in \eqref{InitializationForVariantCentEm} and \eqref{InitializationForCentEm}, given by
\begin{align*}
&\theta^{k+1}(\alpha)=g_1\circ g_2\big(\theta^k(\alpha)\big)\\
&\theta^{k+1}(\rho)=\hat{g}_1\circ \hat{g}_2\big(\theta^k(\rho)\big),\\
\end{align*}
until we found $\theta^\star(\alpha)$ and $\theta^\star(\rho)$ satisfying
\begin{align*}
&\Big\|\theta^\star(\alpha)-g_1\circ g_2\big(\theta^\star(\alpha)\big)\Big\|\leq 10^{-10}\\
&\Big\|\theta^\star(\rho)-\hat{g}_1\circ \hat{g}_2\big(\theta^\star(\rho)\big)\Big\|\leq 10^{-10}.
\end{align*}

The error at iteration $k$ of the distributed algorithms was then computed as
\begin{align*}
&\frac{1}{N}\sum_{n=1}^N \Big\|\pi_1\circ g_1\big((z_n^k(\alpha)\big)-\theta^\star(\alpha)\Big\|\\
&\frac{1}{N}\sum_{n=1}^N \Big\|\pi_1\circ \hat{g}_1\big((z_n^k(\rho)\big)-\theta^\star(\rho)\Big\|,
\end{align*}
where $\pi_1$ is the projection onto the average, \emph{i.e.}, $\pi_1(\mu,p,\sigma^2)=\mu$ (as mentioned before, $p$ and $\sigma^2$ were treated as nuisance parameters).

\

The number of Monte Carlo tests was $100$ and the errors at iteration $k$ were averaged  out of $100$ for each $\alpha$ and $\rho$. The results for two different $\text{SNR}$ values are shown in logarithmic scale in Figures 1 and 2.

\begin{figure}[h!]
\centering
\includegraphics[scale=0.6]{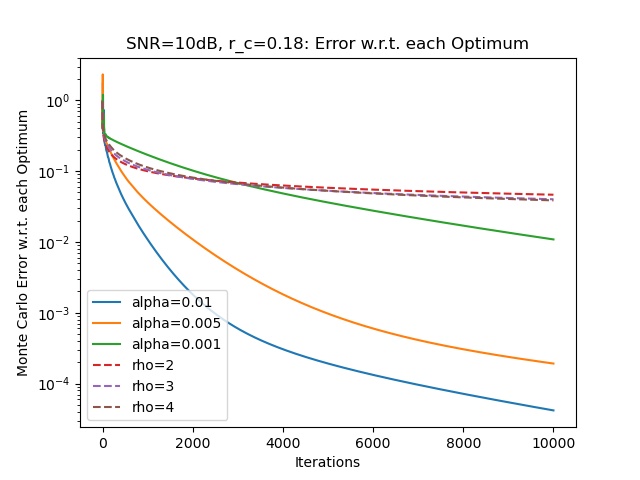}

\caption{The figure shows the result of the Monte Carlo simulation of the error with respect to each optimum for an $\text{SNR}=10$dB and a connectivity radius of $0.18$. The dashed curves correspond to the algorithm from  \cite{pereira2018parameter} with parameter $\rho \in \lbrace 2,3,4\rbrace$ and the non-dashed curves correspond to the DBPI algorithm with parameter $\alpha \in \lbrace 0.001,0.005,0.01\rbrace$.}
\end{figure}

\begin{figure}[h!]
\centering
\includegraphics[scale=0.6]{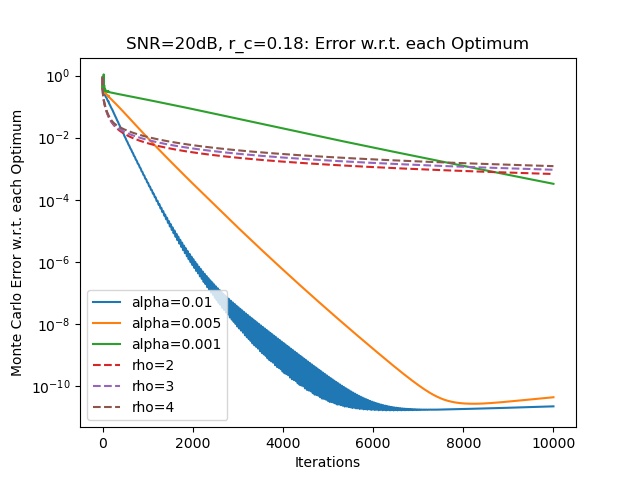}

\caption{The figure shows the result of the Monte Carlo simulation of the error with respect to each optimum for an $\text{SNR}=20$dB and a connectivity radius of $0.18$. The dashed curves correspond to the algorithm from  \cite{pereira2018parameter} with parameter $\rho \in \lbrace 2,3,4\rbrace$ and the non-dashed curves correspond to the DBPI algorithm with parameter $\alpha \in \lbrace 0.001,0.005,0.01\rbrace$.}
\end{figure}

The simulations show, as expected from the theory, that Algorithm 1 converges linearly and clearly outperforms the algorithm from \cite{pereira2018parameter}, which, given its diminishing step-size, is bound to converge only sub-linearly. Moreover, both algorithms require just one round of communications per iteration.

\section{Conclusion}
This article builds upon the distributed Banach-Picard algorithm and its convergence properties provided in  \cite{FranciscodistributedPicardTheory} to make two main contributions: we provided a proof of local linear convergence for the distributed PCA algorithm suggested in \cite{gang2019fast}, thereby filling a gap left by that work; starting from the distributed Banach-Picard iteration, we proposed a distributed algorithm for solving the parameter estimation problem from noisy and faulty measurements that had been  addressed in \cite{pereira2018parameter}. Unlike the algorithm in \cite{pereira2018parameter}, which uses diminishing step sizes, thus exhibiting sublinear convergence rate, the proposed instance of the distributed Banach-Picard iteration is guaranteed to have local linear convergence. Numerical experiments confirm the theoretical advantage of the proposed method with respect to that from \cite{pereira2018parameter}.

\appendices

\section{Regularity Conditions} \label{RegularityConditionSect}
\begin{theorem}\label{th:regularity}
Let $K\subset \Theta$ be a compact set containing $\theta^\star$. Then, for all $\theta \in K$,
\begin{align}\label{PolynomialPartialPhi}
&\Big|\frac{\partial^i \phi}{\partial^{i_1}x_1\ldots\partial^{i_k}x_k}(y,h,\theta)\Big|\leq \mathcal{P}^\phi_{\partial^{i_1}x_1\ldots\partial^{i_k}x_k}(|y|,|h_1|,\ldots,|h_d|)\\\label{PolynomialPartialr}
&\Big|\frac{\partial^i r}{\partial^{i_1}x_1\ldots\partial^{i_k}x_k}(y,h,\theta)\Big|\leq \mathcal{P}^r_{\partial^{i_1}x_1\ldots\partial^{i_k}x_k}(|y|,|h_1|,\ldots,|h_d|),
\end{align}
where $\sum_{j=1}^Ki_j=i\geq 1$, $x_1, \ldots, x_k$ are dummy variables in $\lbrace \mu_1, \ldots, \mu_d, p,\sigma^2\rbrace$ (\emph{i.e.} consider partial diferentiation with respect to the components of $\theta$), and where
\begin{align*}
&\mathcal{P}^r_{\partial^{i_1}x_1\ldots\partial^{i_k}x_k} \quad \text{and} \quad \mathcal{P}^\phi_{\partial^{i_1}x_1\ldots\partial^{i_k}x_k}
\end{align*}
are polynomials.
\end{theorem}

\

We start with a proof of \eqref{PolynomialPartialr}. Note that $r$ satisfies the following differential equations
\begin{align*}
\frac{\partial r}{\partial \mu_i}(y,h,\theta)&=\big(1-r(\theta)\big)r(\theta)\frac{y-h^T\mu}{\sigma^2}\mu_i, \quad i=1, \ldots, d,\\
\frac{\partial r}{\partial p}(y,h,\theta)&=\frac{1}{p(1-p)}r(\theta)\big(1-r(\theta)\big),\\
\frac{\partial r}{\partial \sigma^2}(y,h,\theta)&=\Big(\frac{(y-h^T\mu)^2}{2(\sigma^2)^2}-\frac{y^2}{2(\sigma^2)^2}\Big)r(\theta)\big(1-r(\theta)\big).
\end{align*}
We deduce that
\begin{align} \label{FirstDerivativePolynomialOfR}
\frac{\partial r}{\partial \lambda}(y,h,\theta)=\frac{\tilde{\mathcal{P}}_\lambda^r(r(\theta),y,h,\theta)}{\tilde{\mathcal{Q}}_\lambda^r(p,\sigma^2)},
\end{align}
where $\tilde{\mathcal{P}}_\lambda^r$ and $\tilde{\mathcal{Q}}_\lambda^r$ are polynomials and $\lambda$ is a dummy variable in $\lbrace \mu_1, \ldots, \mu_d,p,\sigma^2\rbrace$.

The chain rule of differentiation, the quotient rule of differentiation and the form \eqref{FirstDerivativePolynomialOfR} imply that
\begin{align}\label{NthDerivativePolynomialOfR}
\frac{\partial^i r}{\partial^{i_1}x_1\ldots\partial^{i_k}x_k}(y,h,\theta)= \frac{\tilde{\mathcal{P}}_{\partial^{i_1}x_1\ldots\partial^{i_k}x_k}^r(r(\theta),y,h,\theta)}{\tilde{\mathcal{Q}}_{\partial^{i_1}x_1\ldots\partial^{i_k}x_k}^r(p,\sigma^2)}.
\end{align}

The result now follows easily from $0<r(\theta)\leq 1$, the compactness of $K$ and the identity \eqref{NthDerivativePolynomialOfR}.

The inequality \eqref{PolynomialPartialPhi} follows from  \eqref{NthDerivativePolynomialOfR} and the form of the gradient of $\phi$ in \eqref{GradientOfPhi}.  This concludes the proof of Theorem \ref{th:regularity}.

\

An immediate corollary of  Theorem \ref{th:regularity} is that the absolute values of the partial derivatives of both $\phi$ and $r$ are dominated by functions whose expectation exists and is finite; this is the content of the following result.

\

\begin{theorem}
Let $\mathcal{P}$ be a polynomial in $d+1$ variables. Then
\begin{align*}
\mathbb{E}_{\theta^\star}\big[\mathcal{P}(|y|,|h_1|,\ldots,|h_d|)\big]
\end{align*}
exists and is finite.
\end{theorem}

\

To prove this theorem, observe that $\mathcal{P}(|y|,|h_1|,\ldots,|h_d|)$ is a sum of elements of the form
\begin{align*}
b|y|^{n_0}|h_1|^{n_1}\ldots|h_d|^{n_d},
\end{align*}
and, hence, it is enough to show that $$\mathbb{E}_{\theta^\star}\big[|y|^{n_0}|h_1|^{n_1}\ldots|h_d|^{n_d}\big]$$ exists and is finite. This last fact is an easy consequence of the existence of absolute non-central and central moments of Gaussians and, therefore, we will skip the proof.

\section{Auxiliary Results and definitions} \label{AuxiliaryAppendix}
\begin{theorem}[\cite{newey1994large}, page 2129]\label{UniformLawOfLargeNumbers}
Let $a(z,\theta)$ be a matrix of functions of an observation $z$ and the parameter $\theta$. If the $z_1, \ldots, z_N$ are i.i.d., $\Omega$ is compact, $a(z_i,\theta)$ is continuous at each $\theta$ and there is $d(z)$ with $\|a(z,\theta)\|_F \leq d(z)$ for all $\theta \in \Omega$, where $\mathbb{E}[d(z)]$ exists and is finite, then $\mathbb{E}[a(z,\theta)]$ is continuous and
\begin{align*}
\sup_{\theta \in \Omega} \Big\|\frac{1}{N}\sum_{j=1}^N a(z_j,\theta)-\mathbb{E}[a(z,\theta)]\Big\|_F \to 0
\end{align*}
in probability.
\end{theorem}

\

Let $X_n$ be a sequence of random vectors. We use the notation $X_n=o_P(1)$, to denote that $X_n$ converges to $0$ in probability, \emph{i.e.}, if, for every $\epsilon>0$, the non-random sequence
\begin{align*}
\mathbb{P}(\|X_n\|\leq \epsilon) 
\end{align*}
converges to $1$. 

If $X_n$ is uniformly bounded in probability, \emph{i.e.}, if, for every $\epsilon>0$, there exists $M(\epsilon)>0$, such that
\begin{align*}
\mathbb{P}\big(\|X_n\|> M(\epsilon)\big)< \epsilon, \quad \forall n,
\end{align*}
we denote this by $X_n=O_P(1)$ (see \cite{van2000asymptotic} for more details and also for the calculus with the $O_P(1)$ and $o_P(1)$).

\section{Proof of Theorem \ref{CrucialLemma}}\label{ProofOfCrucialLemma}
We give only a sketch of the proof of \eqref{CrucialTheoremSecondClaim} (the proof of \eqref{CrucialTheoremFirstClaim} is analogous). Observe that
\begin{align*}
&T_N^\prime(\theta)-T^\prime(\theta)=\\
&\big(A(\theta)\big)^{-1}\Big(A(\theta)-A_N(\theta)\Big)\big(A_N(\theta)\big)^{-1}\frac{1}{N}\sum_{n=1}^N \nabla_\theta^2\phi(y_n,h_n,\theta)\\
&+\big(A(\theta)\big)^{-1}\Big(\frac{1}{N}\sum_{n=1}^N \nabla_\theta^2\phi(y_n,h_n,\theta)-\mathbb{E}_{\theta^\star}\big[\nabla^2_\theta\phi(y,h,\theta)\big]\Big),
\end{align*}
which implies that
\begin{equation} \label{UpperBoundOnTNminusT}
\begin{aligned}
&\|T_N^\prime(\theta)-T^\prime(\theta)\|\leq \big\|A(\theta)\big\|^{-1}\times\\
\Big(&\big\|A(\theta)-A_N(\theta)\big\|\big\|A_N(\theta)\big\|^{-1}\Big\|\frac{1}{N}\sum_{n=1}^N \nabla_\theta^2\phi(y_n,h_n,\theta)\Big\|\\
+&\Big\|\frac{1}{N}\sum_{n=1}^N \nabla_\theta^2\phi(y_n,h_n,\theta)-\mathbb{E}_{\theta^\star}\big[\nabla^2_\theta\phi(y,h,\theta)\big]\Big\|\Big).
\end{aligned}
\end{equation}
From Theorem \ref{UniformLawOfLargeNumbers} (see appendix \ref{AuxiliaryAppendix}),
\begin{align}
\label{CrucialTheoremFirstClaimVariant}
&\sup_{\theta \in \bar{B}^{\|\cdot\|}_{\delta,\theta^\star}}\big\|A_N(\theta)-A(\theta)\big\|_F\to 0, \\ \label{CrucialTheoremSecondClaimVariant}
&\sup_{\theta \in \bar{B}^{\|\cdot\|}_{\delta,\theta^\star}}\big\|\frac{1}{N}\sum_{n=1}^N \nabla_\theta^2\phi(y_n,h_n,\theta)-\mathbb{E}_{\theta^\star}[\nabla^2_\theta\phi(y,h,\theta)]\big\|_F\to 0,
\end{align}
in probability; 
these are  consequences of Theorem \ref{UniformLawOfLargeNumbers}, by noting that:
\begin{description}
\item[a)] $\|\frac{1}{\sigma^2}\Gamma(y,h,\theta)\|_F\leq M\|hh^T\|_F$, where $M$ is the maximum of $\frac{1}{\sigma^2}$ on $\bar{B}^{\|\cdot\|}_{\delta,\theta^\star}$ and where we used the fact that $|r(y,h,\theta)|\leq 1$;
\item[b)] $\|\nabla^2_\theta \phi(y,h,\theta)\|_F\leq g(y,h)$ on $\bar{B}^{\|\cdot\|}_{\delta,\theta^\star}$, for some map $g$ not depending on $\theta$ for which $\mathbb{E}_{\theta^\star}[g(y,h)]$ exists and is finite (see appendix \ref{RegularityConditionSect}).
\end{description}

Since all norms are equivalent, \eqref{CrucialTheoremFirstClaimVariant}-\eqref{CrucialTheoremSecondClaimVariant}  also holds if the Frobenius norm is replaced by any other norm.

Taking the supremum over  on $\bar{B}^{\|\cdot\|}_{\delta,\theta^\star}$ on both sides of \eqref{UpperBoundOnTNminusT}, we obtain, from \eqref{CrucialTheoremFirstClaimVariant}-\eqref{CrucialTheoremSecondClaimVariant}, that 
\begin{align*}
&\sup_{\theta \in \bar{B}^{\|\cdot\|}_{\delta,\theta^\star}}\big\|T_N^\prime(\theta)-T^\prime(\theta)\big\|\\
\leq& o_P(1)\sup_{\theta \in \bar{B}^{\|\cdot\|}_{\delta,\theta^\star}}\|A_N(\theta)\|^{-1}\sup_{\theta \in \bar{B}^{\|\cdot\|}_{\delta,\theta^\star}}\big\|\frac{1}{N}\sum_{n=1}^N \nabla_\theta^2\phi(y_n,h_n,\theta)\big\|\\
+&o_P(1),
\end{align*}
where the definitions of $o_P(1)$ and $O_P(1)$ can be found in appendix \ref{AuxiliaryAppendix}.

From \eqref{CrucialTheoremFirstClaimVariant} and \eqref{CrucialTheoremSecondClaimVariant}, together with the compactness of $\bar{B}^{\|\cdot\|}_{\delta,\theta^\star}$, we can deduce (proof   omitted) that
\begin{align*}
\sup_{\theta \in \bar{B}^{\|\cdot\|}_{\delta,\theta^\star}}\|A_N(\theta)\|^{-1}&=O_P(1)\\
\sup_{\theta \in \bar{B}^{\|\cdot\|}_{\delta,\theta^\star}}\big\|\frac{1}{N}\sum_{n=1}^N \nabla_\theta^2\phi(y_n,h_n,\theta)\big\|&=O_P(1).
\end{align*}

Putting everything together, we conclude that
\begin{align*}
&\sup_{\theta \in \bar{B}^{\|\cdot\|}_{\delta,\theta^\star}}\big\|T_N^\prime(\theta)-T^\prime(\theta)\big\|
\leq\\& o_P(1)O_P(1)O_P(1)+o_P(1)=o_P(1),
\end{align*}
where the equality follows  from the calculus rules with $O_P(1)$ and $o_P(1)$  \cite{van2000asymptotic}.

As mentioned, the proof of \eqref{CrucialTheoremFirstClaim} is entirely analogous; just note that, in order to use Theorem \ref{UniformLawOfLargeNumbers}, we need to check that $\|\nabla_\theta \phi(y,h,\theta)\|\leq \tilde{g}(y,h)$ on $\bar{B}^{\|\cdot\|}_{\delta,\theta^\star}$, for some map $\tilde{g}$ not depending on $\theta$, for which $\mathbb{E}[\tilde{g}(y,z)]$ exists and is finite; this is again a consequence of the results proved in appendix \ref{RegularityConditionSect}.

\section{Proof of Lemma \ref{SangerFixedPointTheorem}} \label{AppendixProofOfLemma}
Suppose $X^\star$ satisfies \eqref{SangerFixedPoint}.  Throughout this proof, $x_i^\star$ denotes the $i$th column of $X^\star$.
Consider the equation imposed by the first column,  $x^\star_1$, \emph{i.e.},
\begin{align*}
Cx^\star_1
=\big((x^\star_1)^TCx^\star_1 \big)x^\star_1,
\end{align*}
and multiply both sides by $(x^\star_1)^T$, which yields
\begin{align*} 
\big((x^\star_1)^TCx^\star_1\big)\big(1-\|x^\star_1\|^2\big)=0.
\end{align*}
From the two  equalities
\begin{align*}
\big((x^\star_1)^TCx^\star_1\big)x^\star_1&=Cx^\star_1 ,\\
\big((x^\star_1)^TCx^\star_1\big)\big(1-\|x^\star_1\|^2\big)&=0,
\end{align*}
we conclude that either $x^\star_1=0$ or $x^\star_1$ is a unit-norm eigenvector of $C$.

Considering the second column, we prove that $x^\star_2$ is either zero or a unit-norm  eigenvector of $C$ that is orthogonal to $x^\star_1$. Observe that
\begin{equation}\label{SecondColumnOfFixedPoint}
\begin{aligned}
Cx^\star_2=
\big((x^\star_1)^TCx^\star_2\big)x^\star_1+\big((x^\star_2)^TCx^\star_2\big)x^\star_2.
\end{aligned}
\end{equation}
Now recall that $x^\star_1=0$ or $x^\star_1$ is a unit-norm eigenvector of $C$. If $x^\star_1=0$, then \eqref{SecondColumnOfFixedPoint} reduces to
\begin{align*}
Cx^\star_2=\big((x^\star_2)^TCx^\star_2\big)x^\star_2
\end{align*}
and the result follows as in the case of $x^\star_1$.
If $x^\star_1 \neq 0$, then it is a unit-norm eigenvector of $C$ and, hence, there exists $\beta$ such that $(x_1^\star)^TC=\beta (x_1^\star)^T$ and \eqref{SecondColumnOfFixedPoint} reduces to
\begin{align} \label{SecondColumnIdentity}
&Cx^\star_2=\beta \big((x^\star_1)^Tx^\star_2\big)x^\star_1+\big((x^\star_2)^TCx^\star_2\big)x^\star_2.
\end{align}

Multiply on the left by $(x^\star_1)^T$ and use  $\|x^\star_1\|^2=1$ to obtain
\begin{align*}
\big((x^\star_2)^TCx^\star_2\big)(x^\star_1)^Tx^\star_2=0.
\end{align*}
If $x^\star_2=0$, we are done. If not, then $0=(x^\star_1)^Tx^\star_2$ and, returning to \eqref{SecondColumnIdentity}, it holds that
\begin{align*}
Cx^\star_2=\big((x^\star_2)^TCx^\star_2\big)x^\star_2.
\end{align*}

This establishes the claim for $x^\star_1$ and $x^\star_2$. Proceeding as we did for the second column, it is possible to construct a proof by induction establishing the result.


%


\section*{Acknowledgment}

This work was partially funded by the Portuguese \textit{Fundação para a Ciência e Tecnologia} (FCT), under grants  PD/BD/135185/2017 and UIDB/50008/2020. The work of João Xavier was supported in part by the Fundação para a Ciência e Tecnologia, Portugal, through the Project LARSyS, under Project FCT Project UIDB/50009/2020 and Project HARMONY PTDC/EEI-AUT/31411/2017 (funded by Portugal 2020 through FCT, Portugal, under Contract AAC n 2/SAICT/2017–031411. IST-ID funded by POR Lisboa under Grant LISBOA-01-0145-FEDER-031411).

\ifCLASSOPTIONcaptionsoff
  \newpage
\fi

\vfill







\begin{thebibliography}{10}
	\providecommand{\url}[1]{#1}
	\csname url@samestyle\endcsname
	\providecommand{\newblock}{\relax}
	\providecommand{\bibinfo}[2]{#2}
	\providecommand{\BIBentrySTDinterwordspacing}{\spaceskip=0pt\relax}
	\providecommand{\BIBentryALTinterwordstretchfactor}{4}
	\providecommand{\BIBentryALTinterwordspacing}{\spaceskip=\fontdimen2\font plus
		\BIBentryALTinterwordstretchfactor\fontdimen3\font minus
		\fontdimen4\font\relax}
	\providecommand{\BIBforeignlanguage}[2]{{%
			\expandafter\ifx\csname l@#1\endcsname\relax
			\typeout{** WARNING: IEEEtran.bst: No hyphenation pattern has been}%
			\typeout{** loaded for the language `#1'. Using the pattern for}%
			\typeout{** the default language instead.}%
			\else
			\language=\csname l@#1\endcsname
			\fi
			#2}}
	\providecommand{\BIBdecl}{\relax}
	\BIBdecl
	
	\bibitem{Saber}
	R.~{Olfati-Saber}, J.~Fax, and R.~Murray, ``Consensus and cooperation in
	networked multi-agent systems,'' \emph{Proceedings of the IEEE}, vol.~95, pp.
	215--233, 2007.
	
	\bibitem{FranciscodistributedPicardTheory}
	F.~Andrade, M.~Figueiredo, and J.~Xavier, ``Distributed {Picard} iteration,''
	\emph{submitted}, available at arXiv:2104.00131, 2021.
	
	\bibitem{pearson1901liii}
	K.~Pearson, ``On lines and planes of closest fit to systems of points in
	space,'' \emph{The London, Edinburgh, and Dublin Philosophical Magazine and
		Journal of Science}, vol.~2, no.~11, pp. 559--572, 1901.
	
	\bibitem{qu2002principal}
	Y.~Qu, G.~Ostrouchov, N.~Samatova, and A.~Geist, ``Principal component analysis
	for dimension reduction in massive distributed data sets,'' in
	\emph{Proceedings of IEEE International Conference on Data Mining (ICDM)},
	vol. 1318, no. 1784, 2002, p. 1788.
	
	\bibitem{liang2014improved}
	Y.~Liang, M.-F.~F. Balcan, V.~Kanchanapally, and D.~Woodruff, ``Improved
	distributed principal component analysis,'' \emph{Advances in Neural
		Information Processing Systems}, vol.~27, pp. 3113--3121, 2014.
	
	\bibitem{kannan2014principal}
	R.~Kannan, S.~Vempala, and D.~Woodruff, ``Principal component analysis and
	higher correlations for distributed data,'' in \emph{Conference on Learning
		Theory}.\hskip 1em plus 0.5em minus 0.4em\relax PMLR, 2014, pp. 1040--1057.
	
	\bibitem{boutsidis2016optimal}
	C.~Boutsidis, D.~P. Woodruff, and P.~Zhong, ``Optimal principal component
	analysis in distributed and streaming models,'' in \emph{Proceedings of the
		forty-eighth annual ACM symposium on Theory of Computing}, 2016, pp.
	236--249.
	
	\bibitem{garber2017communication}
	D.~Garber, O.~Shamir, and N.~Srebro, ``Communication-efficient algorithms for
	distributed stochastic principal component analysis,'' in \emph{International
		Conference on Machine Learning}.\hskip 1em plus 0.5em minus 0.4em\relax PMLR,
	2017, pp. 1203--1212.
	
	\bibitem{bai2005principal}
	Z.-J. Bai, R.~H. Chan, and F.~T. Luk, ``Principal component analysis for
	distributed data sets with updating,'' in \emph{International Workshop on
		Advanced Parallel Processing Technologies}.\hskip 1em plus 0.5em minus
	0.4em\relax Springer, 2005, pp. 471--483.
	
	\bibitem{kargupta2001distributed}
	H.~Kargupta, W.~Huang, K.~Sivakumar, and E.~Johnson, ``Distributed clustering
	using collective principal component analysis,'' \emph{Knowledge and
		Information Systems}, vol.~3, no.~4, pp. 422--448, 2001.
	
	\bibitem{qi2004global}
	H.~Qi, T.-W. Wang, and J.~D. Birdwell, ``Global principal component analysis
	for dimensionality reduction in distributed data mining,'' \emph{Statistical
		data mining and knowledge discovery}, pp. 327--342, 2004.
	
	\bibitem{abu2002distributed}
	F.~N. Abu-Khzam, N.~F. Samatova, G.~Ostrouchov, M.~A. Langston, and A.~Geist,
	``Distributed dimension reduction algorithms for widely dispersed data.'' in
	\emph{IASTED PDCS}, 2002, pp. 167--174.
	
	\bibitem{scaglione2008decentralized}
	A.~Scaglione, R.~Pagliari, and H.~Krim, ``The decentralized estimation of the
	sample covariance,'' in \emph{2008 42nd Asilomar Conference on Signals,
		Systems and Computers}.\hskip 1em plus 0.5em minus 0.4em\relax IEEE, 2008,
	pp. 1722--1726.
	
	\bibitem{le2008distributed}
	Y.-A. Le~Borgne, S.~Raybaud, and G.~Bontempi, ``Distributed principal component
	analysis for wireless sensor networks,'' \emph{Sensors}, vol.~8, no.~8, pp.
	4821--4850, 2008.
	
	\bibitem{yildiz2009distributed}
	M.~E. Yildiz, F.~Ciaramello, and A.~Scaglione, ``Distributed distance
	estimation for manifold learning and dimensionality reduction,'' in
	\emph{2009 IEEE International Conference on Acoustics, Speech and Signal
		Processing}.\hskip 1em plus 0.5em minus 0.4em\relax IEEE, 2009, pp.
	3353--3356.
	
	\bibitem{suleiman2016performance}
	W.~Suleiman, M.~Pesavento, and A.~M. Zoubir, ``Performance analysis of the
	decentralized eigendecomposition and esprit algorithm,'' \emph{IEEE
		Transactions on Signal Processing}, vol.~64, no.~9, pp. 2375--2386, 2016.
	
	\bibitem{korada2011gossip}
	S.~B. Korada, A.~Montanari, and S.~Oh, ``Gossip pca,'' \emph{ACM SIGMETRICS
		Performance Evaluation Review}, vol.~39, no.~1, pp. 169--180, 2011.
	
	\bibitem{li2011distributed}
	L.~Li, A.~Scaglione, and J.~H. Manton, ``Distributed principal subspace
	estimation in wireless sensor networks,'' \emph{IEEE Journal of Selected
		Topics in Signal Processing}, vol.~5, no.~4, pp. 725--738, 2011.
	
	\bibitem{schizas2015distributed}
	I.~D. Schizas and A.~Aduroja, ``A distributed framework for dimensionality
	reduction and denoising,'' \emph{IEEE Transactions on Signal Processing},
	vol.~63, no.~23, pp. 6379--6394, 2015.
	
	\bibitem{wu2017power}
	S.~X. Wu, H.-T. Wai, A.~Scaglione, and N.~A. Jacklin, ``The power-oja method
	for decentralized subspace estimation/tracking,'' in \emph{2017 IEEE
		International Conference on Acoustics, Speech and Signal Processing
		(ICASSP)}.\hskip 1em plus 0.5em minus 0.4em\relax IEEE, 2017, pp. 3524--3528.
	
	\bibitem{wu2018review}
	S.~X. Wu, H.-T. Wai, L.~Li, and A.~Scaglione, ``A review of distributed
	algorithms for principal component analysis,'' \emph{Proceedings of the
		IEEE}, vol. 106, no.~8, pp. 1321--1340, 2018.
	
	\bibitem{Gang}
	A.~Gang, B.~Xiang, and W.~Bajwa, ``Distributed principal subspace analysis for
	partitioned big data: Algorithms, analysis, and implementation,'' \emph{IEEE
		Transactions on Signal and Information Processing over Networks}, vol.~7, pp.
	699--715, 2021.
	
	\bibitem{dimakis2010gossip}
	A.~G. Dimakis, S.~Kar, J.~M. Moura, M.~G. Rabbat, and A.~Scaglione, ``Gossip
	algorithms for distributed signal processing,'' \emph{Proceedings of the
		IEEE}, vol.~98, no.~11, pp. 1847--1864, 2010.
	
	\bibitem{xiao2006distributed}
	J.-J. Xiao, A.~Ribeiro, Z.-Q. Luo, and G.~B. Giannakis, ``Distributed
	compression-estimation using wireless sensor networks,'' \emph{IEEE Signal
		Processing Magazine}, vol.~23, no.~4, pp. 27--41, 2006.
	
	\bibitem{boyd2006randomized}
	S.~Boyd, A.~Ghosh, B.~Prabhakar, and D.~Shah, ``Randomized gossip algorithms,''
	\emph{IEEE transactions on information theory}, vol.~52, no.~6, pp.
	2508--2530, 2006.
	
	\bibitem{barbarossa2007decentralized}
	S.~Barbarossa and G.~Scutari, ``Decentralized maximum-likelihood estimation for
	sensor networks composed of nonlinearly coupled dynamical systems,''
	\emph{IEEE Transactions on Signal Processing}, vol.~55, no.~7, pp.
	3456--3470, 2007.
	
	\bibitem{zhao2007information}
	T.~Zhao and A.~Nehorai, ``Information-driven distributed maximum likelihood
	estimation based on gauss-newton method in wireless sensor networks,''
	\emph{IEEE Transactions on Signal Processing}, vol.~55, no.~9, pp.
	4669--4682, 2007.
	
	\bibitem{schizas2007consensus}
	I.~D. Schizas, A.~Ribeiro, and G.~B. Giannakis, ``Consensus in ad hoc wsns with
	noisy links—part i: Distributed estimation of deterministic signals,''
	\emph{IEEE Transactions on Signal Processing}, vol.~56, no.~1, pp. 350--364,
	2007.
	
	\bibitem{stankovic2010decentralized}
	S.~S. Stankovi{\'c}, M.~S. Stankovic, and D.~M. Stipanovic, ``Decentralized
	parameter estimation by consensus based stochastic approximation,''
	\emph{IEEE Transactions on Automatic Control}, vol.~56, no.~3, pp. 531--543,
	2010.
	
	\bibitem{sayed2014diffusion}
	A.~H. Sayed, ``Diffusion adaptation over networks,'' in \emph{Academic Press
		Library in Signal Processing}.\hskip 1em plus 0.5em minus 0.4em\relax
	Elsevier, 2014, vol.~3, pp. 323--453.
	
	\bibitem{pereira2018parameter}
	S.~S. Pereira, R.~L{\'o}pez-Valcarce, and A.~Pages-Zamora, ``Parameter
	estimation in wireless sensor networks with faulty transducers: A distributed
	{EM} approach,'' \emph{Signal Processing}, vol. 144, pp. 226--237, 2018.
	
	\bibitem{faria2010fitting}
	S.~Faria and G.~Soromenho, ``Fitting mixtures of linear regressions,''
	\emph{Journal of Statistical Computation and Simulation}, vol.~80, no.~2, pp.
	201--225, 2010.
	
	\bibitem{gang2019fast}
	A.~Gang, H.~Raja, and W.~U. Bajwa, ``Fast and communication-efficient
	distributed {PCA},'' in \emph{IEEE International Conference on Acoustics,
		Speech, and Signal Processing (ICASSP)}, 2019, pp. 7450--7454.
	
	\bibitem{sanger1989optimal}
	T.~D. Sanger, ``Optimal unsupervised learning in a single-layer linear
	feedforward neural network,'' \emph{Neural networks}, vol.~2, no.~6, pp.
	459--473, 1989.
	
	\bibitem{shi2015extra}
	W.~Shi, Q.~Ling, G.~Wu, and W.~Yin, ``Extra: An exact first-order algorithm for
	decentralized consensus optimization,'' \emph{SIAM Journal on Optimization},
	vol.~25, no.~2, pp. 944--966, 2015.
	
	\bibitem{Gang2021}
	A.~Gang and W.~Bajwa, ``A linearly convergent algorithm for distributed
	principal component analysis,'' available at arXiv:2101.01300, 2021.
	
	\bibitem{mixtures_book}
	G.~{McLachlan} and D.~Peel, \emph{Finite Mixture Models}.\hskip 1em plus 0.5em
	minus 0.4em\relax John Wiley \& Sons, 2004.
	
	\bibitem{Dempster}
	A.~Dempster, N.~Laird, and D.~Rubin, ``Maximum likelihood from incomplete data
	via the {EM} algorithm,'' \emph{Journal of the Royal Statistical Society
		(Series B)}, vol.~29, pp. 1--–37, 1977.
	
	\bibitem{kar2008distributed}
	S.~Kar and J.~M. Moura, ``Distributed consensus algorithms in sensor networks
	with imperfect communication: Link failures and channel noise,'' \emph{IEEE
		Transactions on Signal Processing}, vol.~57, no.~1, pp. 355--369, 2008.
	
	\bibitem{nedic2009distributed}
	A.~Nedic and A.~Ozdaglar, ``Distributed subgradient methods for multi-agent
	optimization,'' \emph{IEEE Transactions on Automatic Control}, vol.~54,
	no.~1, pp. 48--61, 2009.
	
	\bibitem{redner1984mixture}
	R.~A. Redner and H.~F. Walker, ``Mixture densities, maximum likelihood and the
	{EM} algorithm,'' \emph{SIAM review}, vol.~26, no.~2, pp. 195--239, 1984.
	
	\bibitem{sundberg1974maximum}
	R.~Sundberg, ``Maximum likelihood theory for incomplete data from an
	exponential family,'' \emph{Scandinavian Journal of Statistics}, vol.~1,
	no.~2, pp. 49--58, 1974.
	
	\bibitem{balakrishnan2017statistical}
	S.~Balakrishnan, M.~Wainwright, and B.~Yu, ``Statistical guarantees for the
	{EM} algorithm: From population to sample-based analysis,'' \emph{Annals of
		Statistics}, vol.~45, no.~1, pp. 77--120, 2017.
	
	\bibitem{nowak2003distributed}
	R.~D. Nowak, ``Distributed em algorithms for density estimation and clustering
	in sensor networks,'' \emph{IEEE transactions on signal processing}, vol.~51,
	no.~8, pp. 2245--2253, 2003.
	
	\bibitem{forero2011distributed}
	P.~A. Forero, A.~Cano, and G.~B. Giannakis, ``Distributed clustering using
	wireless sensor networks,'' \emph{IEEE Journal of Selected Topics in Signal
		Processing}, vol.~5, no.~4, pp. 707--724, 2011.
	
	\bibitem{xiao2004fast}
	L.~Xiao and S.~Boyd, ``Fast linear iterations for distributed averaging,''
	\emph{Systems \& Control Letters}, vol.~53, no.~1, pp. 65--78, 2004.
	
	\bibitem{magnus2019matrix}
	J.~R. Magnus and H.~Neudecker, \emph{Matrix Differential Calculus with
		Applications in Statistics and Econometrics}.\hskip 1em plus 0.5em minus
	0.4em\relax John Wiley \& Sons, 2019.
	
	\bibitem{lutkepohl1997handbook}
	H.~Lutkepohl, \emph{Handbook of Matrices}.\hskip 1em plus 0.5em minus
	0.4em\relax Wiley, 1996.
	
	\bibitem{bishop2006pattern}
	C.~Bishop, \emph{Pattern Recognition and Machine Learning}.\hskip 1em plus
	0.5em minus 0.4em\relax Springer, 2006.
	
	\bibitem{mclachlan2007algorithm}
	G.~J. McLachlan and T.~Krishnan, \emph{The EM algorithm and Extensions}.\hskip
	1em plus 0.5em minus 0.4em\relax John Wiley \& Sons, 2007, vol. 382.
	
	\bibitem{horn2012matrix}
	R.~Horn and C.~Johnson, \emph{Matrix Analysis}.\hskip 1em plus 0.5em minus
	0.4em\relax Cambridge University Press, 2012.
	
	\bibitem{bertsekas2004stochastic}
	D.~Bertsekas and S.~Shreve, \emph{Stochastic Optimal Control: the Discrete-time
		Case}.\hskip 1em plus 0.5em minus 0.4em\relax Athena Scientific, 1996.
	
	\bibitem{borwein2010convex}
	J.~Borwein and A.~S. Lewis, \emph{Convex Analysis and Nonlinear Optimization:
		Theory and Examples}.\hskip 1em plus 0.5em minus 0.4em\relax Springer, 2010.
	
	\bibitem{van2000asymptotic}
	A.~Van~der Vaart, \emph{Asymptotic Statistics}.\hskip 1em plus 0.5em minus
	0.4em\relax Cambridge University Press, 2000, vol.~3.
	
	\bibitem{ortega2000iterative}
	J.~M. Ortega and W.~C. Rheinboldt, \emph{Iterative Solution of Nonlinear
		Equations in Several Variables}.\hskip 1em plus 0.5em minus 0.4em\relax SIAM,
	2000.
	
	\bibitem{newey1994large}
	W.~K. Newey and D.~McFadden, ``Large sample estimation and hypothesis
	testing,'' \emph{Handbook of Econometrics}, vol.~4, pp. 2111--2245, 1994.
	
\end{thebibliography}
\end{document}